\documentclass{article}
\usepackage{graphicx} 
\usepackage{amsmath, amssymb, amsthm}
\usepackage{graphicx}
\usepackage{bm}
\usepackage{upgreek}
\usepackage{enumitem} 
\usepackage{hyperref}
\hypersetup{colorlinks,linkcolor={blue},citecolor={red},urlcolor={blue}}
\usepackage{cleveref}
\usepackage{mathtools}
\usepackage{enumitem}

\newtheorem*{assumptionA*}{\textbf{Assumptions}\hspace{-3pt}}
\Crefname{assumptionA}{\textbf{Assumptions}\hspace{-3pt}}{\textbf{H}\hspace{-3pt}}
\crefname{assumptionA}{\textbf{Assumptions}}{\textbf{A}}

\newtheorem{thm}{Theorem}
\newtheorem{defn}{Definition}

\newtheorem{lemma}{Lemma}

\newtheorem{prop}{Proposition}

\newtheorem{assmp}{Assumption}

\newtheorem{note}{Note}
\newtheorem{rem}{Remark}
\usepackage[letterpaper,top=2cm,bottom=2cm,left=4cm,right=4cm,marginparwidth=1.75cm]{geometry}

\usepackage{graphicx}
\graphicspath{ {./Diagrams/} }
\usepackage{float}

\title{A Lyapunov-Tamed Euler Method for Singular SDEs}
\author{Tim Johnston, Pierre Monmarché}
\date{January 2026}

\begin{document}
\maketitle
\begin{abstract}
Many applications, such as systems of interacting particles in physics, require the simulation of diffusion processes with singular coefficients. Standard Euler schemes are then not convergent, and theoretical guarantees in this situation are scarce. In this work we introduce a Lyapunov-tamed Euler scheme, for drift coefficients for which the weak derivative is dominated by a function that obeys a certain generic Lyapunov-type condition. This allows for a range of coefficients that explode to infinity on a bounded set. We establish that, in terms of $L^p$-strong error, the Lyapunov-tamed scheme is consistent and moreover achieves the same order of convergence as the standard Euler scheme for Lipschitz coefficients. The general result is applied to systems of mean-field particles with singular repulsive interaction in 1D, yielding an error bound with polynomial dependency in the number of particles.
\end{abstract}



\section{Introduction}
Theoretical bounds for numerical approximations of SDEs have long since relied on smoothness and regularity assumptions, which ensure that the behavior of the true solution does not change too radically between grid-points, and additionally that the moments of the approximation and the true solution are suitably bounded (one may consult \cite{kloeden2013numerical} for a classic book length study). In particular, both of these assumptions are covered by ensuring that the coefficients of an SDE obey a Lipschitz assumption. In the absence of a Lipschitz assumption, even in the case where the drift coefficient $b:\mathbb{R}^d\to\mathbb{R}^d$ obeys a coercivity-type condition 
\begin{equation}\label{eq: coercivity}
    \langle b(x),x\rangle\leq c(1+\lvert x\rvert^2),
\end{equation}
which ensures the polynomial moments of the true solution are bounded, one can obtain bad performance of explicit schemes such as the Euler scheme. In particular, as shown in \cite{articlediv, ae1c5b69-79b7-3bf4-87cb-ba6dd72c5ca2}, there is a small probability (for all step-sizes) that the scheme blows up astronomically, therefore leading to a divergence in $L^p$ for the scheme as the stepsize tends to $0$. Since the moments of the scheme diverge as the stepsize tends to $0$, accurate approximation of the true solution in $L^p$ is impossible. This insight led to the development of `tamed' Euler schemes in \cite{da2b46b6-84fa-3eed-be6d-20c143c60df1}, in which the drift coefficient $b$ is replaced in the definition of the scheme by a new coefficient $b_n$ which depends on the stepsize. A large number of works have extended and developed this idea, for instance the alternative taming approaches to \cite{da2b46b6-84fa-3eed-be6d-20c143c60df1} considered in \cite{articletamed, MAO2015370}, schemes that tame also the diffusion coefficient in \cite{articledif}, higher order schemes in \cite{SABANIS201984}, sampling algorithms in \cite{BROSSE20193638, 10.1093/imanum/drad038}, SPDEs in \cite{0329d33d261740b59da99f06a80a9570, Brehier2022}, particle systems for McKean-Vlasov SDEs in \cite{10.1214/21-AAP1760}, SDEs driven by Lévy noise in \cite{doi:10.1137/151004872}, and schemes that preserve exponential stability in \cite{09794cb6-f51d-34af-8144-efe133f5a106}. Furthermore, a very general tamed method for approximating SDEs with non-Lipschitz coefficients was proposed in \cite{arnulf2012}, however the generality of this approach meant that no rate of convergence could be guaranteed (and indeed pathological cases as studied in \cite{doi:10.1080/07362994.2016.1263157, Hairer2012LossOR} show that in such general cases a rate of convergence is unlikely).

In the majority of these works the superlinearly growing drift coefficient obeys a bound of the form
\begin{equation}\label{eq: tamed assmp}
    \lvert b(x)-b(y)\rvert\leq L(1+\lvert x\rvert^l+\lvert y\rvert^l)\lvert x-y\rvert.
\end{equation}
for some $l,L>0$. After making such an assumption, given sufficient moment bounds on both the scheme and the true solution it is possible to prove convergence of the Euler scheme using a variation of the arguments one uses in the case of Lipschitz coeffients.

However, less attention has been given to SDEs whose coefficients explode to infinity on a compact 
subset of $\mathbb{R}^d$, which is the focus of this article. This class of examples is especially important in modeling particles, in particular in statistical physics and in molecular simulations, where the repulsion forces can typically diverge as two particles get close to one another. One may consult \cite{29acd3d494044594aea0829ef236aad6} for example for a book-long study from a mathematical perspective. In particular, in the mean-field regime, at the continuous-time level, these models have been the topic of a number of works over recent years, with important progresses in establishing quantitative propagation of chaos as the number of particles goes to infinity in the presence of singular forces, including the 2D vortex model \cite{jabin2018quantitative,guillin2024uniform,monmarche2024time,wang2025sharp} or Coulomb and Riesz gaz \cite{Serfaty,de2023sharp,Guillin2022OnSO} and also second-order systems~\cite{2024arXiv240204695B}. Tamed algorithms for mean-field systems with non-globally Lipschitz  coefficients have recently been considered in \cite{2025arXiv251016427S,chen2025wellposedness,tran2025infinite} (see also references within, in particular the literature comparison in Table~1 of \cite{2025arXiv251016427S}), but focusing on conditions in the spirit of~\eqref{eq: tamed assmp}, covering coefficients growing faster than linearly at infinity but not singular cases. For singular forces, the weak convergence of a class of Euler-type schemes was studied in \cite{refId0} for the underdamped Langevin dynamics.

In \cite{JOHNSTON2026104772} a tamed Euler scheme was given for drift coefficients that are not locally integrable, specifically SDEs taking values in an open set $D\subset \mathbb{R}^d$ for which Lipschitz constant of the drift coefficient $b$ blows up on $\partial D$. In the present work we generalise the assumptions of \cite{JOHNSTON2026104772} to a more general and natural `Lyapunov-type' condition, namely that the drift coefficient satisfies $b:D\to\mathbb{R}^d$ satisfies
\begin{equation}\label{eq: Lyapunov Lipschitz b}
      \lvert  b(x)- b(y)\rvert \leq  ( V(x)+ V(y))\lvert x-y \rvert,
\end{equation}
for some function $V:D\to[0,\infty)$ for which $V(x)\to\infty$ as $x\to\partial D$. This condition implies that the weak derivative of $\nabla b$ is bounded by $V$, see the calculation in \eqref{eq: nabla b bound}. It follows that our main main object of interest is the following SDE taking values in an open set $D\subset \mathbb{R}^d$ with additive noise, that is
\begin{equation}\label{eq: SDE main}
    dX_t=b(X_t)dt+\beta dW_t, \;\;\;X_0=x_0\in D,
\end{equation}
where $\beta>0$ is a constant which controls the intensity of the noise. Therefore, given control of moments of the scheme and the true solution composed with $V$, one may control crucial errors arising from the discretisation. See also \cite{chak2024regularity} for a similar study under a condition of the form~\eqref{eq: Lyapunov Lipschitz b}, although set on the whole space $\mathbb R^d$ (i.e. without singularities).

In considering such coefficients, which can be in a sense `discontinuous' if considered as functions on the whole of $\mathbb{R}^d$, this work is part of a large family of works extended the analysis of numerical methods for SDEs to pathological cases, including for discontinuous coefficients as recently considered by a number of authors, see \cite{Leobacher2016multi, articlemg, 10.1214/20-EJP479, 10.1214/22-AAP1867}, with matching lower bounds in \cite{nonliplowerbound, 10.1214/22-AAP1837,ELLINGER2024101822}. This fast moving area is covered in the survey paper \cite{mullergronbach2024complexity}. However, our techniques and philosophy are much closer to the tamed Euler scheme literature than to the approach used for discontinuous coefficients.

We apply our formalism to two main examples: firstly in Section \ref{sec: singular ips} we show that  this scheme can be used for systems of interacting particles with singular interaction kernel, which was the key example that instigated this study. In this direction we show that a class of interacting particle systems can be approximated at rate $1$ with polynomial dependence on the number of particles. As far as we are aware this is the first theoretical result on the strong approximation of such a system. Furthermore, in Section \ref{sec: recovering tamed schemes} we show that our scheme can easily cover a standard setting often used  in the tamed Euler scheme literature, where $b$ obeys a polynomial Lipschitz assumption, as in \eqref{eq: tamed assmp}.
\section{Assumptions and Main Results}
\subsection{Assumptions}
In order that our approach be fruitful, it is necessary that the dynamics of \eqref{eq: SDE main} preserve moments of $V(X_t)$ on a finite time interval. Furthermore, it turns out we shall need an amount of `give' to account for the errors induced by the discretisation. To this end, let us recall the generator of \eqref{eq: SDE main} as the operator $\mathcal{L}:C^2(D)\to C(D)$ given as
\begin{equation}\label{eq: generator defn}
\mathcal{L} K :=\langle \nabla K, b\rangle +\frac{\beta}{2}\Delta K.
\end{equation}
Then our first assumption is as follows
\begin{assmp}\label{assmp: Lyapunov fn}
There exists an open set $D\subset \mathbb{R}^d$ and functions $b\in C(\mathbb{R}^d,\mathbb{R}^d)$ and $V\in C^2(D,[0,\infty))$ such that \eqref{eq: Lyapunov Lipschitz b} holds for every $x,y \in D$. Furthermore, for every $p\geq 2$ there exist $a_p,b_p>0$ and $q_p>1$ such that 
\begin{equation}\label{eq: generator assmp}
\mathcal{L} (V^p) +\frac{1}{q_p}\lvert \nabla (V^p)\rvert^{q_p}\leq a_p +b_pV^p,
\end{equation}
where $\mathcal{L}$ is as given in \eqref{eq: generator defn}. We assume these constants additionally satisfy $(a_q)^{p/q}\leq a_p$ and $(b_q)^{p/q}\leq b_p$ for $2\leq p<q$, and that $q_p>1$ is decreasing in $p\geq 2$. Furthermore, for every sequence $x_n\in D$ such that $x_n\to x\in \partial D$ one has
\begin{equation}\label{eq: explosion on boundary}
    V(x_n)\to \infty.
\end{equation}
Additionally, one can bound $\lvert \nabla V\rvert$ by a power of $V$, that is there exists $C,l>0$ such that
\begin{equation}\label{eq: nabla V power bound}
\lvert \nabla V(x)\rvert \leq C(1+V(x)^l).
\end{equation}
Finally, there exists $\mu\in \mathbb{R}$ such that the following monotonicity assumption holds
\begin{equation}\label{eq: mon assump}
\langle b(x)-b(y),x-y\rangle\leq \mu \lvert x-y\rvert^2.
\end{equation}
\end{assmp}

\begin{rem}
    The second term on the LHS of \eqref{eq: generator assmp} enforces a limit on the growth of $V$, and therefore $b$. In particular it rules out coefficients that grow exponentially at infinity, since in such a context this second term will be of higher order than any negative term given in the main generator expression.
\end{rem}
The assumption that the Lyapunov function $V$ obeys \eqref{eq: Lyapunov Lipschitz b} shall be crucial for the increment bounds required for our analysis. The assumption \eqref{eq: generator assmp} ensures that moments of $V(X_t)$ and $V(X^n_t)$ can be appropriately bounded, with the second term on the LHS corresponding to errors from the discretization. The assumption in \eqref{eq: explosion on boundary} and \eqref{eq: nabla V power bound} are simply used to ensure that a strong solution exists for the continuous time dynamics. Finally \eqref{eq: mon assump} is a standard assumption used to show an $L^p$ convergence result for the scheme via a Gröwall type argument. However, since this implies a coercivity type condition \eqref{eq: coercivity}, in our case this assumption will also be necessary for proving moment bounds, hence why we have not written it as a seperate assumption.

Finally the reason we have assumed the constants $a_p,b_p>0$ scale in a particular way with $p>0$ is so that we can collate them in a generic constant with the right scaling properties. Particularly one has
\begin{defn}\label{defn: generic constant}
Let $c>0$ denote  a constant that depends only on $T,p,\mu>0$, i.e. the length of the interval over which the numerical simulations are considered, and the $L^p$ space in which moment and convergence bounds are proven. We also assume $c>0$ depends on the convexity parameter $\mu>0$, for technical reasons due to how we calculate moment and convergence bounds, and since this shall not change when we consider families of interacting particle systems. Furthermore, for every $q>0$ we denote by $H_q>0$ a constant of the form
\begin{equation}\label{eq: H_p defn}
    H_q:=c(1+\lvert b(x_0)\rvert^q+V(x_0)^q+\lvert x_0\rvert^q+\beta^{q}d^{q/2} +a_q)e^{b_q T},
\end{equation}
where $c>0$ is as above.
\end{defn}
The assumption on $a_p,b_p>0$ then ensures for $p>q>0$ one has
\begin{equation}\label{eq: constant condition}
    (H_q)^{p/q}\leq H_p,\;\;\;H_pH_{q}\leq H_{p+q},
\end{equation}
where the second part follows from Young's inequality. Furthermore, if one denotes $q_p'>1$ such that
\begin{equation}\label{eq: qp' defn}
   1/ q_p'+1/q_p=1,
\end{equation}
we shall see that the quantity $q_p'>1$ plays a crucial role, and therefore under the assumption $q_p$ is decreasing in $p>0$ one shall obtain that $q_p'$ is increasing in $p>0$. In the example of the singular interacting particle system considered in Section \ref{sec: singular ips}, one has that $q_{p}'$ is an affine function of $p$. 

Our second assumption is a smoothness assumption which pertains to the gradient of $b$. Under this assumption we shall be able to upgrade the rate of convergence from $1/2$ to $1$, reflecting results in the Lipschitz and superlinear case, see for example \cite{articletamed}.
\begin{assmp}\label{assmp: Lyapunov controls nabla b}
Let $b\in C^1(D)$. Suppose there exists $\hat{V}\in C^2(D)$ which satisfied Assumption \ref{assmp: Lyapunov fn} for new constants $\hat{a}_p, \hat{b}_p, \hat{q}_p>0$, and for which the following improved smoothness condition additionally holds
\begin{equation}\label{eq: higher order Lyapunov}
  \lvert b(x)-b(y)- \nabla b(x)(x-y) \rvert \leq c (1+ \hat{V}(x)+ \hat{V}(y))\lvert x-y \rvert^2.
\end{equation} 
\end{assmp}
For the constants $\hat{a}_p, \hat{b}_p, \hat{q}_p>0$ given in Assumption \ref{assmp: Lyapunov controls nabla b} we introduce a new generic constant
\begin{defn}\label{defn: generic constant 2}
    For $c>0$ as given in Definition \ref{defn: generic constant}, one defines a new generic constant
       \begin{equation}\label{eq: H_p hat defn}
    \hat{H}_q:=c(1+\lvert b(x_0)\rvert^q+V(x_0)^q+\lvert x_0\rvert^q+\beta^{q}d^{q/2} +\hat{a}_q)e^{\hat{b}_q T}.
\end{equation}
\end{defn}
The constants $H_q, \hat{H}_q$ shall play a crucial role when we consider the interacting particle system, since we shall be interested in following their dependence on the number of particles $N$.
\subsection{Tamed Scheme}
As with other tamed schemes in the literature, we shall introduce new coefficients that depend on the stepsize, so as to `curb' the behavior of the scheme in the cases where the drift coefficient would otherwise take very large values. However, in our case we do not hide the fact that our scheme will be nonsensical in these rare cases, and merely rely on the fact that such pathological trajectories will be extremely rare for small stepsize (see \eqref{eq: tau_n defn} and Lemma \ref{lemma: decay of stopping time}). For this reason, the scheme may not be as practical as others with the same theoretical properties, however we do not bother to address this since the focus of this article is on theoretical convergence.

Our new coefficient $b_n:\mathbb{R}^d\to\mathbb{R}^d$, corresponding to a stepsize $n^{-1}$, will then be given as
\begin{equation}\label{eq: truncated coef}
b_n(x):=b(x)\cdot 1_{x\in D}\cdot 1_{V(x)\leq n^{1/2}}\cdot 1_{\lvert b(x)\rvert \leq n^{1/2}},
    \end{equation}
so as to be truncated when $V(x)$ or $b(x)$ takes too large a value, or when $x\not \in D$. The associated tamed Euler scheme is therefore given as
\begin{equation}\label{eq: tamed discrete}
   x^n_{m-1}=x^n_{m-1}+n^{-1}b_n(x_n)+n^{-1/2}\beta z_m, \;\;\;x^n_0=x_0\in D,
\end{equation}
where $(z_m)_{m\geq 1}$ is a sequence of $\mathcal{N}(0,I_d)$ random variables. However, for convenience, and in order to study the strong rate of convergence, we shall consider the continuous time lift \begin{equation}\label{eq: scheme defn}
      dX^n_t:=b_n(X^n_{\kappa_n(t)})dt+\beta dW_t,\;\;\; X^n_0=x_0\in D
    \end{equation}
     where $\kappa_n(t):=\frac{\lfloor nt\rfloor}{n}$ is the backwards projection onto the grid $\{0,1/n,2/n,...\}$, and $X^n_t$ is driven by the same noise as the continuous time process $X_t$ given in \eqref{eq: SDE main}.
     \subsection{Main Results}
     Our two main general convergence theorems are thus given as follows
     \begin{thm}\label{thm: conv of scheme}
        Let Assumption \ref{assmp: Lyapunov fn} hold. Then the continuous time SDE \eqref{eq: SDE main} has a unique strong solution $X_t$ for $t\in[0,\infty)$, and satisfies $X_t\in D$ for every $t>0$ almost surely. Furthermore, for every $T>0$ and $p\geq2$ there exists $c>0$ such that for every $n \geq V(x_0)^2$, the tamed scheme \eqref{eq: scheme defn} satisfies 
         \begin{align}
       \sup_{t\in [0,T]} E\lvert X^n_t-X_t\rvert^p \leq H_{24p+q_{12p}'}n^{-p/2},
    \end{align}
        for the generic constant $H_{(\cdot)}$ introduced in Definition \ref{defn: generic constant}, and for $q_{10p}'>0$ given by 
        \eqref{eq: qp' defn} and Assumption \ref{assmp: Lyapunov fn}.
     \end{thm}
     The bounds given in Theorem \ref{thm: conv of scheme} are conservative in order to clean up the analysis. However, the fact that the order of the constant in front of the rate depends on the conjugate of the constant $q_p>0$ given in Assumption \ref{assmp: Lyapunov fn} seems like an unalterable part of our analysis (see \eqref{eq: scheme V Ito}). In addition, the condition that $n \geq V(x_0)^2$ ensures that the scheme starts in the reasonable regime where $b_n(x_0)=b(x_0)$. 

     With the greater smoothness assumption of Assumption \ref{assmp: Lyapunov controls nabla b}, one is able to obtain convergence at rate $1$.
     \begin{thm}\label{thm: conv of scheme 2}
         Let Assumption \ref{assmp: Lyapunov controls nabla b} hold. Then \eqref{eq: SDE main} has a unique strong solution in $D$ as before, and furthermore for every $n \geq V(x_0)^3$ and $p\geq 4$ one has
           \begin{align}
       \sup_{t\in [0,T]} E\lvert X^n_t-X_t\rvert^p \leq (H_{24p+q'_{12p}}+\hat{H}_{7p+q'_{5p}})n^{-p},
    \end{align}
    where $\hat{H}_{(\cdot)}$ is the generic constant given in Definition \ref{defn: generic constant 2}.
     \end{thm}
     \begin{rem}
The fact that the dependence on the generic constant $H_{(\cdot)}$ is the same in both main Theorems is an artifact of various simplifications in the calculations. If all dependencies were optimised the dependence in Theorem \ref{thm: conv of scheme} would be less severe than the dependence in Theorem \ref{thm: conv of scheme 2}.
     \end{rem}
     \subsection{Properties of Tamed Coefficient}
Before proving moment bounds of the scheme and the true solution, it shall be useful to derive some key properties from the assumptions given in the previous section. All the following bounds are given under Assumption \ref{assmp: Lyapunov fn}, and so may be easily applied given variant constants under Assumption \ref{assmp: Lyapunov controls nabla b}.

First of all, let us prove growth bounds on $b(x)$. In these bounds a `generic' point is required from which the Lipschitz-type assumption \eqref{eq: Lyapunov Lipschitz b} can be applied (in the standard tamed literature it may be that $0$ is chosen for this role, however in our case we don't necessarily expect $0\in D$). However, since we wish to consider families of SDEs with $N\in \mathbb{N}$ interacting particles, we must be careful that there are no unintended dependencies. Therefore, although it may be suboptimal, we pick the initial condition $x_0\in D$ as this `generic' point. One notes immediately that Assumption \ref{assmp: Lyapunov fn} implies 
    \begin{align}
       \lvert b(x)\rvert &\leq \lvert b(x_0)\rvert+\lvert b(x)-b(x_0)\rvert \nonumber \\
       &\leq \lvert b(x_0)\rvert+(1+V(x_0)+V(x))\lvert x-x_0\rvert,
    \end{align}
    and therefore since $\lvert b(x_0)\rvert \leq H_1$ and $V(x_0)\leq H_1$ by the definition of the generic constant in Definition \ref{defn: generic constant}, one has
    \begin{align}\label{eq: first b bound}
       \lvert b(x)\rvert \leq H_1(1+V(x))(1+\lvert x \rvert).
    \end{align}
    Therefore, for every $q>0$ it follows
    \begin{equation}\label{eq: upper bound for b}
 \lvert b(x)\rvert ^q\leq H_q (1+\lvert x \rvert^q)(1+V(x)^q),
    \end{equation}
    and so by Young's inequality
    \begin{align}\label{eq: upper bound for b 2}
       \lvert b(x)\rvert^q \leq H_q(1+\lvert x \rvert^{2q}+V(x)^{2q}).
    \end{align}
    Now we shall use the monotonicity assumption \eqref{eq: mon assump} to obtain 
     \begin{align}
     \langle x, b(x)\rangle & = \langle x-x_0, b(x)-b(x_0)\rangle+\langle x_0, b(x)-b(x_0)\rangle+\langle x, b(x_0)\rangle \nonumber \\
     &\leq \mu\lvert x-x_0\rvert^2+\lvert x_0\rvert(\lvert b(x)\rvert+\lvert b(x_0)\rvert)+\lvert b(x_0)\rvert\lvert x\rvert.
    \end{align}
    Using Young's inequality, and recalling the generic constant $c>0$ depends on $\mu$, and that $\lvert b(x_0)\rvert\leq H_1$, one may furthermore obtain
     \begin{align}\label{eq: b bound using mon}
     \langle x, b(x)\rangle &\leq c\lvert x \rvert^2+H_1\lvert  b(x)\rvert +H_2\nonumber \\
     &\leq c\lvert x \rvert^2+H_1(1+\lvert x \rvert)(1+V(x))+H_2.
    \end{align}
Now let us find a bound for $\nabla b$ under Assumption \ref{assmp: Lyapunov controls nabla b}. Since we therefore have that $b\in C^1(D)$, and since this implies Assumption \ref{assmp: Lyapunov fn}, one simply calculates for every $x\in D$
\begin{equation}\label{eq: nabla b bound}
   \lvert  \nabla b(x)\rvert = \lim_{y\to x}\frac{\lvert b(y)-b(x)\rvert}{\lvert y-x\rvert}\leq 2V(x).
\end{equation}

\subsection{Sketch of Proof}
As was discussed in the literature some time before the development of tamed schemes, see \cite{0b5e028f-1a33-3fed-ae08-229fd5443c3f, MATTINGLY2002185}, the main problem with Euler approximations for non-globally Lipschitz coefficients is that the moments of the scheme can explode under certain conditions. Actually, in the case we consider, it shall be possible that the regular Euler scheme always has infinite moments, since the drift coefficient $b$ does not have to be locally integrable. Therefore the main challenge shall be proving moment bounds for the tamed Euler scheme we construct. After this is achieved, convergence can be shown by a standard Grönwall-type argument familiar in the numerical analysis literature. To this end our strategy is as follows
\begin{itemize}
    \item we introduce the following stopping time
 \begin{equation}\label{eq: tau_n defn}
        \tau_n:=\inf\{t\geq 0\vert \;\max\{V(X^n_t), \lvert b(X^n_t)\rvert\}\geq n^{1/2}\},
    \end{equation}
    which reflects the first time that $b_n(X^n_t)\not=b(X^n_t)$.
    \item we prove moment bounds of all orders for the Euler scheme and $V(X^n_t)$ up until the stopping time $\tau_n$
    \item we then use this to show that the probability $\tau_n\in [0,T]$ decays quickly in $n$. Specifically, for all $p>0$, by Markov's inequality one has
    \begin{align}
        &P(\tau_n\in [0,T])=P(\sup_{t\in [0,T]} \max\{V(X^n_t), \lvert b(X^n_t)\rvert\}\geq n^{1/2})\nonumber \\
        &=P(\sup_{t\in [0,T]} \max\{V(X^n_{t\wedge \tau_n})^p, \lvert b(X^n_{t\wedge \tau_n})\rvert^p\}\geq n^{p/2})\nonumber \\
&\leq n^{-p/2}E\sup_{t\in [0,T]} \max\{V(X^n_{t\wedge \tau_n})^p, \lvert b(X^n_{t\wedge \tau_n})\rvert^p\}.
    \end{align}
    \item therefore, so long as the taming we implement means that no trajectory of $X^n_t$ and $V(X^n_t)$ diverges too severely with $n\geq 1$, one may show uniform moment bounds of both over the entirety of $[0,T]$. This is the precisely the calculation in \eqref{eq: splitting argument}
\end{itemize}
The tool we use to upgrade `uniform-in-expectation' to `expectation-of-uniform' moment bounds\footnote{sometimes referred to a `weak' and `strong' moment bounds} is the Langlart domination inequality, see Proposition \ref{prop: Langlart ineq}.
\section{Moment Bounds}
The proof of the convergence of $X^n_t$ to $X_t$ will be relatively standard once we show the moments of both processes, as well as `$V$-moments' (that is, moments of both processes composed with the Lyapunov function $V$). For convenience, let us collate the information of the following section into one proposition
\begin{prop}\label{prop: mom bounds}
    Let Assumption \ref{assmp: Lyapunov fn} hold.  Then \eqref{eq: SDE main} has a unique strong solution, and for every $T>0$ and $p\geq 2$ one has
    \begin{equation}
        \sup_{t\in [0,T]} EV(X_t)^p\leq H_p,\;\;\; \sup_{t\in [0,T]}  E\lvert X_t \rvert ^p\leq H_{2p},
    \end{equation}
\begin{equation}
    \sup_{t\in [0,T]}EV(X^n_{t\wedge \tau_n} )^p\leq H_{p+q_p'},\;\;\;\sup_{t\in [0,T]}E\hat{V}(X^n_{t\wedge \tau_n} )^p\leq \hat{H}_{p+\hat{q}_p'},
\end{equation}
    \begin{equation}
     \sup_{t\in [0,T]} E\lvert X^n_{t\wedge \tau_n} \rvert^p\leq H_{2p+q_p'},\;\;\; \sup_{t\in [0,T]} E\lvert X^n_t \rvert^p\leq H_{12p+q_{6p}'}.
    \end{equation}
    where $\tau_n$ is the stopping time given in \eqref{eq: tau_n defn}.
\end{prop}
This result will be proven in the lemmas of the following section. In particular, the bound on $\hat{V}(X^n_{t\wedge \tau_n} )^p$ follows as an exact parallel of the bound on $V(X^n_{t\wedge \tau_n} )^p$.


\subsection{Continuous Process}
We begin with moments for the continuous time process, which are much simpler.
\subsubsection{V-Moments}
\begin{lemma}\label{lem: V bounds cont non-uniform}
    Let Assumption \ref{assmp: Lyapunov fn} hold. Then \eqref{eq: SDE main} has a unique strong solution, and for every $T>0$ and $p\geq 2$ one has
    \begin{equation}
\sup_{t\in [0,T]} EV(X_t)^p\leq H_p.
    \end{equation}
\end{lemma}
\begin{proof}
  Firstly, let us define
  \begin{equation}
      dX^R_t=b(X^R_t)1_{\{\lvert b(X^R_t)\rvert \leq R,\;\lvert V(X^R_t)\rvert \leq R\}
      }dt+\beta dW_t.
  \end{equation}
  Note, since the drift is bounded, that this SDE has a unique strong solution following Theorem 1.1 in \cite{ZHANG20051805}. Now let us define the stopping time
   \begin{equation}\label{eq: result of V moment bound cont}
       \tilde{\tau}_R:=\inf\{t\geq 0\vert \; V(X^R_t)\geq R\}.
   \end{equation}
Applying the Itô's formula in Proposition \ref{prop: local Ito} for $t\in [0,T]$, one obtains that
  \begin{align}\label{eq: Ito's formula V}
   V(X^R_{t\wedge \tilde{\tau}_R})^p&=V(x_0)+\int^{t\wedge \tilde{\tau}_R} _0 \mathcal{L}(V^p)(X^R_s) ds\nonumber \\
   &+\int^{t\wedge \tilde{\tau}_R} _0 p\beta V^{p-1}(X^R_s)\langle \nabla V(X^R_s),dW_s\rangle.
  \end{align}  
  Note that $V(X^R_{s\wedge \tilde{\tau}_R})$ and $\nabla V(X^R_{t\wedge \tilde{\tau}_R})$ are almost surely bounded, by the definition of stopping time given above, and by \eqref{eq: nabla V power bound}. Therefore, the stochastic integral has bounded integrand, and therefore is a martingale, and so applying expectation and \eqref{eq: generator assmp}, one has
    \begin{align}
        E V(X^R_{t\wedge \tilde{\tau}_R})^p&\leq V(x_0)^p+E\int^{t\wedge \tilde{\tau}_R}_0 (a_p+b_pV(X^R_s)^p)ds.
    \end{align}
   Then since $V\geq 0$ it follows
        \begin{align}\label{eq: V power p cont}
        E V(X^R_{t\wedge \tilde{\tau}_R})^p\leq V(x_0)^p+a_pt+b_p\int^t_0 EV(X^R_{s\wedge \tilde{\tau}_R})^p ds.
    \end{align}
    Applying the almost sure boundedness of the integrand again, and additionally Grönwall's inequality, one obtains
     \begin{align}\label{eq: bound with stopping time}
        EV(X^R_{t\wedge \tilde{\tau}_R})^p&\leq (V(x_0)^p+a_pt)e^{b_p t}.
    \end{align}
Observe for $R\geq V(x_0)$ one has
  \begin{equation}
      R^pP(\tilde{\tau}_R\in [0,T])\leq  EV(X^R_{T\wedge \tilde{\tau}_R})^p,
  \end{equation}
 so that $P(\tilde{\tau}_R\in [0,T])\to 0$ as $R\to \infty$. Furthermore, since $\tilde{\tau}_R$ is increasing in $R>0$, one has that $\lim_{R\to\infty}\tilde{\tau}_R>T$ almost surely. Therefore, since $X^R_t$ is a unique strong solution of \eqref{eq: SDE main} up to $\tilde{\tau}_R$, the process $(X_t)_{t\geq 0}$ given as $X_t:=\lim_{R\to\infty}X^R_{t\wedge \tilde{\tau}_R}$ is well defined for all $t\geq 0$, and yields a unique strong solution of \eqref{eq: SDE main} on $[0,T]$. We may therefore extend this to a solution on $[0,\infty)$ by stitching together solutions on compact intervals. Taking the limit as $R\to\infty$ in \eqref{eq: bound with stopping time} and applying Fatou's lemma, the stated result follows form Definition \ref{defn: generic constant}.
\end{proof}

\begin{lemma}\label{lem: V hat bounds cont uniform}
    Let Assumption \ref{assmp: Lyapunov fn} and Assumption \ref{assmp: Lyapunov controls nabla b} hold. Then for every $T>0$ and $p\geq 2$ one has
    \begin{equation}
\sup_{t\in [0,T]} E\hat{V}(X_t)^p\leq \hat{H}_p.
    \end{equation}
\end{lemma}
\begin{proof}
    Follows exactly as the previous two lemmas, using Assumption \ref{assmp: Lyapunov controls nabla b}.
\end{proof}
\subsubsection{Regular Moments}
\begin{lemma}\label{lem: mom bounds cont non-uniform}
    Let Assumption \ref{assmp: Lyapunov fn} hold. Then for all $T>0$ and $p\geq 2 $ 
    \begin{equation}
       \sup_{0\leq t\leq T}  E\lvert X_t \rvert ^p\leq H_{2p}.
    \end{equation}
\end{lemma}
\begin{proof}
        As before, we introduce a stopping time to assume apriori that all stochastic integrals vanish under expectation. Specifically, let
    \begin{equation}
        \bar{\tau}_R:=\inf\{t\geq 0 \vert \; \lvert X_t\rvert \geq R \}.
    \end{equation}
    We fix $p\geq 2$ as assumed in the Lemma statement. Then one calculates
    \begin{equation}\label{eq: calc for second derivative}
       tr(\nabla^2\lvert \cdot\rvert^p)(x)=p(p+d-2)\lvert x\rvert^{p-2}\leq cd\lvert x\rvert^{p-2}, 
    \end{equation}
        where $c>0$ is the generic constant introduced in Definition \ref{defn: generic constant}, which in particular depends on $p>0$. Therefore, for $t\in [0,T]$ one may apply Itô's formula to yield
        \begin{align}\label{eq: Ito's formula cont}
  E\lvert X_{t \wedge \bar{\tau}_R} \rvert^p &\leq \lvert x_0\rvert^p+cE\int^{t \wedge \bar{\tau}_R}_0 (\lvert X_s\rvert^{p-2}\langle X_s, b(X_s)\rangle+\beta^2 d\lvert X_s\rvert^{p-2})ds,
    \end{align}
    Now we use \eqref{eq: b bound using mon} to obtain that
    \begin{equation}\label{eq: inner product calculation}
        \lvert x \rvert^{p-2}\langle b(x),x \rangle\leq c(1+\lvert x \rvert^p)+H_p(1+V(x)^p)+H_p,
    \end{equation}
    so that by Young's inequality, and since $\beta^pd^{p/2}\leq H_p$ and $\lvert x_0\rvert^p\leq H_p$ by Definition \ref{defn: generic constant}, one has
    \begin{align}
  E\lvert X_{t \wedge \bar{\tau}_R} \rvert^p &\leq \lvert x_0\rvert^p+cE\int^{t \wedge \bar{\tau}_R}_0 (\lvert X_s\rvert^p+H_p(1+V( X_s)^p)+H_p)ds\nonumber \\
  &\leq c\int^{t}_0 (E\lvert X_{s \wedge \bar{\tau}_R}\rvert^p+H_p(1+EV( X_{s \wedge \bar{\tau}_R})^p)+H_p)ds\nonumber \\
  &\leq c\int^{t}_0 (E\lvert X_{s \wedge \bar{\tau}_R}\rvert^p+H_{2p})ds,
    \end{align}
    where we have used the properties of the generic constant given in \eqref{eq: constant condition}. Therefore, by Grönwall's lemma and Fatou's lemma the result follows.
\end{proof}

\subsection{Tamed Euler Scheme}
We begin by proving `bad' moment bounds that scale with the step-size parameter $n\geq 1$.
\begin{lemma}\label{lem: bad mom bounds scheme}
   Let Assumption \ref{assmp: Lyapunov fn} hold. Then for all $T>0$ and $p\geq 2$ there exists $c>0$ such that
    \begin{equation}
      \sup_{t\in [0,T]}E\lvert X^n_t\rvert^p\leq H_pn^{p/2}.
    \end{equation}
\end{lemma}
\begin{proof}
    Similarly to in Lemma \ref{lem: mom bounds cont non-uniform}, let us consider the following stopping time
    \begin{equation}
     \tilde{\tau}_n^R:=\inf\{t\geq 0\vert \; \lvert X^n_t\rvert\geq R\}. 
    \end{equation} 
    Let $p\geq 2$ as in the Lemma statement. By Itô's formula, the definition of $b_n$ in \eqref{eq: truncated coef}, and since by this and the stopping time definition the stochastic integral is a martingale, one obtains for $t\in [0,T]$ that
    \begin{align}\label{eq: Ito for stopped}
  &E\lvert X^n_{t \wedge \tilde{\tau}_n^R} \rvert^{p}\leq \lvert x_0\rvert^p+E\int^{t  \wedge \tilde{\tau}_n^R}_0 (p\lvert X^n_s\rvert^{p-2}\langle X^n_s, b_n(X^n_{\kappa_n(s)}\rangle +c\beta^2 d\lvert X^n_s \rvert^{p-2})ds \nonumber \\
  &\leq \lvert x_0\rvert^p+cE\int^{t  \wedge \tilde{\tau}_n^R}_0 (n^{p/2} +\lvert X^n_s \rvert^p+ \beta^pd^{p/2} )ds.
    \end{align}
   Then applying Grönwall's lemma and Fatou's lemma, and since $\beta^pd^{p/2}\leq H_p$, the result follows.
    \end{proof}

\begin{lemma}\label{lem: V bounds scheme stopped}
    Let Assumption \ref{assmp: Lyapunov fn} hold.  Then for every $T>0$ and $p\geq 2$ one has
    \begin{equation}
       \sup_\tau \sup_{t\in [0,T]}EV(X^n_{t\wedge \tau\wedge \tau_n} )^p\leq H_{p+q_p'},
    \end{equation}
    where the first supremum is over all stopping times $\tau$ and $\tau_n$ is given in \eqref{eq: tau_n defn}.
\end{lemma}
\begin{proof}
    Let $\tau$ be an arbitrary stopping time. We begin by proving an increment bound up to $\tau\wedge \tau_n$. Particularly, since one has 
    \begin{equation}
    b_n(X^n_{t\wedge \tau_n})=b(X^n_{t\wedge \tau_n}),\;\;\; \lvert b_n(X^n_{t\wedge \tau_n}) \rvert\leq n^{1/2},\;\;\;V(X^n_{t\wedge \tau_n})\leq n^{1/2},
    \end{equation}
    by the definition of $\tau_n$ and $b_n$, one first applies \eqref{eq: Lyapunov Lipschitz b} to obtain
    \begin{align}
       & \lvert b_n(X^n_{t\wedge \tau\wedge \tau_n})-b_n(X^n_{\kappa_n(t)\wedge\tau\wedge \tau_n})\rvert = \lvert b(X^n_{t\wedge \tau\wedge \tau_n})-b(X^n_{\kappa_n(t)\wedge\tau\wedge \tau_n})\rvert \nonumber \\
       & \leq (V(X^n_{t\wedge \tau\wedge \tau_n})+V(X^n_{\kappa_n(t)\wedge\tau\wedge \tau_n}))[n^{-1}b(X^n_{\kappa_n(t)\wedge\tau\wedge \tau_n})+\beta \lvert W_{t\wedge \tau\wedge \tau_n}-W_{\kappa_n(t)\wedge\tau\wedge \tau_n}\rvert]\nonumber\\
        &\leq c(1+\beta \sup_{u\in [\kappa_n(t),t]}\lvert W_u-W_{\kappa_n(t)}\rvert]).
    \end{align}
    Therefore, applying standard facts about Brownian motion, and since $\beta^pd^{p/2}\leq H_p$ one has
      \begin{align}\label{eq: increment bound}
       & E\lvert b_n(X^n_{t\wedge \tau\wedge \tau_n})-b_n(X^n_{\kappa_n(t)\wedge\tau\wedge \tau_n})\rvert^p\leq H_{p}.
    \end{align}
    Now let $p\geq 2$. Note that by the definition of $\tau_n$ and \eqref{eq: nabla V power bound}, the stochastic integral is a true  martingale and so vanishes under expectation by the optional stopping theorem. Therefore, applying Itô's formula and correcting the discretisation via Young's inequality, one has for $t\in [0,T]$ that
    \begin{align}\label{eq: scheme V Ito}
        &EV(X^n_{t\wedge \tau\wedge \tau_n})^p=V(x_0)^p+E\int^{t\wedge \tau \wedge \tau_n}_0 (\mathcal{L}(V^p))(X^n_s)ds\nonumber \\
        &+E\int^{t\wedge \tau \wedge \tau_n}_0\langle \nabla (V^p)(X^n_s), b_n(X^n_{\kappa_n(s)})-b_n(X^n_s)\rangle  ds\nonumber \\
        &\leq V(x_0)^p+E\int^{t\wedge \tau \wedge\tau_n}_0  ((\mathcal{L}(V^p))(X^n_s)+\frac{1}{q_p}\lvert \nabla (V^p)(X^n_s)\rvert^{q_p})ds\nonumber \\
        &+E\int^{t\wedge \tau \wedge\tau_n}_0 \frac{1}{q_p'}\lvert b_n(X^n_{\kappa_n(s)})-b_n(X^n_s)\rvert^{q_p'} ds.
    \end{align}
    Note that $1/q_p'$ is necessarily bounded by a universal constant, so that one obtains via \eqref{eq: generator assmp}
    \begin{align}
        EV(X^n_{t\wedge \tau\wedge \tau_n})^p& \leq V(x_0)^p+E\int^{t\wedge \tau \wedge\tau_n}_0  a_p +b_pV(X^n_s)^pds\nonumber \\
        &\quad +\int^t_0 cE\lvert b_n(X^n_{\kappa_n(s)\wedge \tau \wedge\tau_n})-b_n(X^n_{s\wedge \tau \wedge\tau_n})\rvert^{q_p'} ds\nonumber \\
        &\leq V(x_0)^p+H_{q_p'}+E\int^{t\wedge \tau \wedge\tau_n}_0  (a_p +b_pV(X^n_s)^p)ds,
    \end{align}
    at which point the result follows from Grönwall's lemma. The dependence on the generic constant follows since $V(x_0)^p\leq H_p$ and since by \eqref{eq: constant condition} one has 
    \begin{equation}
        H_p+H_{q_p'}\leq (H_p+1)(H_{q_p'}+1)\leq H_pH_{q_p'}\leq H_{p+q_p'}.
    \end{equation}
\end{proof}

Now we may prove uniform moments bounds of $X^n_{t\wedge \tau_n}$ in a similar fashion.
\begin{lemma}\label{lem: scheme bounds stopped non-uniform}
      Let Assumption \ref{assmp: Lyapunov fn} hold.  Then for every $T>0$ and $p\geq 2$ one has
    \begin{equation}
    \sup_{\tau}\sup_{t \in [0,T]}E\lvert X^n_{t\wedge \tau\wedge \tau_n}\rvert^p\leq H_{2p+q_p'},
    \end{equation}
    where the first supremum is taken over all stopping times.
\end{lemma}

\begin{proof}
    Let $p\geq 2$. Repeating \eqref{eq: Ito for stopped} with \eqref{eq: inner product calculation} and applying \eqref{eq: increment bound}, as well as recalling $\lvert x_0\rvert^p\leq H_p$, one obtains via Lemma \ref{lem: V bounds scheme stopped} that for $t\in [0,T]$
     \begin{align}
  &E\lvert X^n_{t \wedge\tau\wedge\tau_n \wedge \tilde{\tau}_n^R} \rvert^{p}\leq \lvert x_0\rvert^p+E\int^{t \wedge\tau\wedge\tau_n \wedge \tilde{\tau}_n^R}_0 (p\lvert X^n_s\rvert^{p-2}\langle X^n_s, b_n(X^n_{s})\rangle +c\beta^2 d\lvert X^n_s \rvert^{p-2})ds \nonumber \\
  &+E\int^t_0 c(\lvert X^n_{s\wedge\tau\wedge\tau_n \wedge \tilde{\tau}_n^R}\rvert^p+\lvert b_n(X^n_{s\wedge\tau\wedge\tau_n \wedge \tilde{\tau}_n^R})-b_n(X^n_{\kappa_n(s)\wedge\tau\wedge\tau_n \wedge \tilde{\tau}_n^R}\rvert^p)ds\nonumber \\
  &+E\int^t_0 H_p(1+V(X^n_{s\wedge\tau\wedge\tau_n \wedge \tilde{\tau}_n^R}))ds\nonumber \\
  &\leq \int^{t }_0 c(E\lvert X^n_{s \wedge\tau_n \wedge \tilde{\tau}_n^R}\rvert^p+H_{2p+q_p'})ds.
    \end{align}
    Therefore the result follows by applying Grönwall's inequality and then Fatou's lemma.
\end{proof}
\subsubsection{Uniform Moment Bounds}
Now let us use the Langlart domination inequality to show that with large probability the stopping time $\tau_n$ given in \eqref{eq: tau_n defn} satisfies $\tau_n\geq T$.

\begin{lemma}\label{lemma: decay of stopping time}
      Let Assumption \ref{assmp: Lyapunov fn} hold.  Then for every $T>0$ and $p\geq 2$ one has
    \begin{equation}
      P(\tau_n\in [0,T])\leq H_{6p+q_{3p}'}n^{-p/2}.
    \end{equation}
\end{lemma}
\begin{proof}
    By the relevant definitions and Markov's inequality, as well as \eqref{eq: upper bound for b 2}, one has
    \begin{align}
    &P(\tau_n\in [0,T])=P(\sup_{t\in [0,T]}\max\{V(X^n_t), \lvert b(X^n_t)\rvert\}>n^{1/2}) \nonumber \\
    &\leq P(\sup_{t\in [0,T]}\max\{V(X^n_t)^p, \lvert b(X^n_t)\rvert^p\}>n^{p/2})\nonumber \\
    &\leq cn^{-p/2}(E\sup_{t\in [0,T]}V(X^n_{t\wedge \tau_n})^p+E\sup_{t\in [0,T]}\lvert b(X^n_{t\wedge \tau_n})\rvert^p)\nonumber \\
    &\leq cn^{-p/2}(E\sup_{t\in [0,T]}V(X^n_{t\wedge \tau_n})^p+H_p(E\sup_{t\in [0,T]}(V(X^n_{t\wedge \tau_n})^{2p}+\lvert X^n_{t\wedge \tau_n}\rvert^{2p}))).
    \end{align}
Now let us show that this expectation may be bounded. Specifically, by Lemma \ref{lem: V bounds scheme stopped} note that for every stopping time 
\begin{equation}
\sup_{t\in [0,T]}EV(X^n_{t\wedge\tau\wedge \tau_n})^{3p}\leq H_{3p+q_{3p}'},\;\;\;\sup_{t\in [0,T]}E\lvert X^n_{t\wedge\tau\wedge \tau_n}\rvert^{3p}\leq H_{6p+q_{3p}'},
\end{equation}
so taking $Z^1_t:=H_{q_{3p}'}$ and $Z^2_t:=H_{3p}$ as constant processes, one may conclude
\begin{equation}
  E\sup_{t\in [0,T]}(V(X^n_{t\wedge \tau_n})^{2p}+\lvert X^n_{t\wedge \tau_n}\rvert^{2p})\leq H_{6p+q_{3p}'},
\end{equation}
at which point the result follows by the properties given in \eqref{eq: constant condition} for the generic constant.
\end{proof}
Now we may prove uniform moment bounds for the scheme over the whole interval
\begin{lemma}\label{lem: moment bounds scheme uniform full intevral}
    Let Assumption \ref{assmp: Lyapunov fn} hold.  Then for every $T>0$ and $p\geq 2$ one has
    \begin{equation}
      \sup_{\tau}\sup_{t\in [0,T]} E\lvert X^n_t \rvert^p\leq H_{12p+q_{6p}'}.
    \end{equation}
\end{lemma}
\begin{proof}
    One writes for $t\in [0,T]$
    \begin{align}\label{eq: splitting argument}
     \lvert X^n_t \rvert^p&=\lvert X^n_t \rvert^p1_{\tau_n\leq T} +\lvert X^n_t \rvert^p1_{\tau_n\geq T}\nonumber \\
      &\leq \lvert X^n_t \rvert^p1_{\tau_n\leq T} +\lvert X^n_{t\wedge \tau_n} \rvert^p.
    \end{align}
    Therefore by Lemma \ref{lem: bad mom bounds scheme}, Lemma \ref{lemma: decay of stopping time} and Lemma \ref{lem: scheme bounds stopped non-uniform}
    \begin{align}
      E\lvert X^n_t \rvert^p&=c\sqrt{E\lvert X^n_t \rvert^{2p}P(\tau_n\leq T)} +cE\lvert X^n_{t\wedge \tau_n} \rvert^p\nonumber \\
      &\leq \sqrt{H_{4p+q'_{2p}}n^{p}H_{12p+q_{6p}'}n^{-p}}+H_{2p+q_p'}\nonumber \\
      &\leq H_{12p+q_{6p}'}.
    \end{align}
     Here we have used several properties of the generic constant. In particular, since $q_p>1$ is decreasing in $p>0$, its conjugate $q_p'>1$ is decreasing in $p>0$. Furthermore, since $H_p$ is increasing in $p>0$ (up to the other generic constant $c=c(p,T,\mu)>0$), one has $H_{2p}\leq H_{12p+q_{6p}'}$.
\end{proof}

\section{Proof of Convergence}
In this section we prove Theorem \ref{thm: conv of scheme} and Theorem \ref{thm: conv of scheme 2}.
\subsection{Proof of Theorem \ref{thm: conv of scheme}}
We begin by splitting the $L^p$ error between the scheme and the true solution as
    \begin{align}\label{eq: final splitting}
\lvert X_t-X^n_t\rvert^p&\leq c\lvert X_{t\wedge \tau_n}-X^n_{t\wedge \tau_n}\rvert^p+c\lvert X^n_{t\wedge \tau_n}-X^n_t\rvert^p\nonumber \\
&+c\lvert X_{t\wedge \tau_n}-X_t\rvert^p\nonumber \\
&=:I^1_t+I^2_t+I^3_t.
    \end{align}
To control the first term $I^1_t$, let us assume $p\geq 2$ and apply the Itô's formula in Proposition \ref{prop: local Ito}. Then by the monotonicity assumption~\eqref{eq: mon assump}, and the fact $b_n(X^n_{t\wedge \tau_n})=b(X^n_{t\wedge \tau_n})$ by construction, one obtains
 \begin{align}
\lvert X_{t\wedge \tau_n}-X^n_{t\wedge \tau_n}\rvert^p&\leq \int^{t\wedge \tau_n}_0 p\langle b(X_s)-b(X^n_{\kappa_n(s)}),X_s-X^n_s\rangle ds\nonumber \\
&\leq \int^t_0 c( \lvert X_{s\wedge \tau_n}-X^n_{s\wedge \tau_n} \rvert^p+\lvert b(X^n_{s\wedge \tau_n})-b(X^n_{\kappa_n(s)\wedge \tau_n})\rvert^p) ds.
    \end{align}
Therefore we must control the increment term. To this end, let us apply Assumption \ref{assmp: Lyapunov fn}, the definition of $\tau_n$ in \eqref{eq: tau_n defn}, and Hölder's inequality to obtain
     \begin{align}
       &E\lvert b(X^n_{s\wedge \tau_n})-b(X^n_{\kappa_n(s)\wedge \tau_n})\rvert^p  \nonumber\\
       &\leq cE(1+V(X^n_{s\wedge \tau_n})^p+V(X^n_{\kappa_n(s)\wedge \tau_n})^p)(n^{-p}\lvert b(X^n_{\kappa_n(s)\wedge \tau_n})\rvert^p+\beta^p\lvert W_{s\wedge \tau_n}-W_{\kappa_n(s)\wedge \tau_n} \rvert^p).
    \end{align}
    Now we may apply the bound $\lvert b(X^n_{\kappa_n(s)\wedge \tau_n})\rvert\leq n^{1/2}$, that follows from the definition of $\tau_n$, to obtain via Proposition \ref{prop: mom bounds}
         \begin{align}\label{eq: inc bound}
       &E\lvert b(X^n_{s\wedge \tau_n})-b(X^n_{\kappa_n(s)\wedge \tau_n})\rvert^p  \nonumber\\
       &\leq c(1+\sup_{s\in [0,T]}EV(X^n_{s\wedge \tau_n})^p)n^{-p}\nonumber \\
       &+c\beta^p\sqrt{(1+\sup_{s\in [0,T]}EV(X^n_{s\wedge \tau_n})^{2p})E\lvert W_{s\wedge \tau_n}-W_{\kappa_n(s)\wedge \tau_n} \rvert^{2p}}\nonumber \\
       &\leq H_{4p+q_{2p}'}n^{-p/2}.
    \end{align}
    Applying this, and furthermore Gronwall's inequality (since the process has finite moments by Proposition \ref{prop: mom bounds}), one obtains that 
    \begin{equation}
    \sup_{t\in [0,T]}EI^1_t\leq H_{24p+q_{12p}'}n^{-p/2}
    \end{equation}
    uniformly in $t\in [0,T]$. 
    
    For $I^2_t$ observe that for $t\in [0,T]$ one may write
    \begin{equation}
        I^2_t=c\lvert X^n_{t\wedge \tau_n}-X^n_t\rvert^p1_{\tau_n\leq T},
    \end{equation}
    since in the case $\tau_n\geq T$ the expression above vanishes. Then applying Lemma \ref{lemma: decay of stopping time} one has for $t\in [0,T]$ that
    \begin{align}
        EI^2_t&=c(E\lvert X^n_{t\wedge \tau_n}-X^n_t\rvert^{2p})^{1/2}P(
        \tau_n\leq T)^{1/2}\nonumber \\
        &\leq  \sqrt{H_{12p+q'_{6p}}}\sqrt{H_{12p+q'_{6p}}n^{-p}} \nonumber \\
        &\leq H_{24p+q'_{12p}}n^{-p/2},
    \end{align}
    by the properties of the generic constant in \eqref{eq: constant condition}, and the fact that $q'_p$ is increasing in $p>0$. A similar argument yields the same bound for $I^3_t$.

\subsection{Proof of Theorem \ref{thm: conv of scheme 2}}
For convenience let us denote
\begin{equation}
    e_t:=X_t-X^n_t.
\end{equation}
Let us now split as
    \begin{align}\label{eq: new splitting}
&\lvert e_{t\wedge \tau_n}\rvert^p\leq \int^{t\wedge  \tau_n}_0 p\langle b(X_{\kappa_n(s)})-b(X^n_{\kappa_n(s)}),\lvert e_{\kappa_n(s)}\rvert^{p-2}e_{\kappa_n(s)}\rangle ds\nonumber \\
&+\int^{t\wedge  \tau_n}_0 p\langle b(X_s)-b(X_{\kappa_n(s)})-\nabla b(X_{\kappa_n(s)})(X_s-X_{\kappa_n(s)}),\lvert e_{\kappa_n(s)}\rvert^{p-2}e_{\kappa_n(s)}\rangle ds \nonumber \\
&+\int^{t\wedge \tau_n}_0 p\langle \nabla b(X_{\kappa_n(s)})(X_s-X_{\kappa_n(s)}),\lvert e_{\kappa_n(s)}\rvert^{p-2}e_{\kappa_n(s)}\rangle ds \nonumber \\
&+\int^{t\wedge \tau_n}_0 p\langle b(X_{\kappa_n(s)})-b(X^n_{\kappa_n(s)}),\lvert e_s\rvert^{p-2}e_s-\lvert e_{\kappa_n(s)}\rvert^{p-2}e_{\kappa_n(s)}\rangle ds\nonumber \\
&=:I^{11}_t+I^{12}_t+I^{13}_t+I^{14}_t.
    \end{align}
Similarly to before, one uses the monotonicity assumption \eqref{eq: mon assump} to bound $I^{11}_t$ as
    \begin{align}\label{eq: I^1_t1}
  EI^{11}_t&\leq   c \int^t_0 E\lvert e_{\kappa_n(s)\wedge\tau_n}\rvert^pds.
    \end{align}
Furthermore, using Assumption \ref{assmp: Lyapunov controls nabla b} one may bound
\begin{align}
    &I^{12}_t\leq c \int^{t\wedge\tau_n}_0 ((1+\hat{V}(X_s)^p+\hat{V}(X_{\kappa_n(s)})^p)\lvert X_s-X_{\kappa_n(s)}\rvert^{2p}+\lvert e_{\kappa_n(s)}\rvert^p)ds \nonumber \\
    &\leq c \int^t_0 ((1+\hat{V}(X_{s\wedge  \tau_n})^p+\hat{V}(X_{\kappa_n(s)\wedge  \tau_n})^p)\lvert X_{s\wedge  \tau_n}-X_{\kappa_n(s)\wedge \tau_n}\rvert^{2p}+\lvert e_{\kappa_n(s)\wedge  \tau_n }\rvert^p)ds .
\end{align}
Then using \eqref{eq: upper bound for b} one writes
\begin{align}
   &(1+\hat{V}(X_{s\wedge  \tau_n})^p+\hat{V}(X_{\kappa_n(s)\wedge  \tau_n})^p)\lvert X_{s\wedge  \tau_n}-X_{\kappa_n(s)\wedge \tau_n}\rvert^{2p} \nonumber \\
   &\leq n^{-2p}H_{2p}(1+\hat{V}(X_{s\wedge  \tau_n})^{5p}+\hat{V}(X_{\kappa_n(s)\wedge  \tau_n})^{5p}+V(X_{\kappa_n(s)\wedge \tau_n})^{5p}+\lvert X_{\kappa_n(s)\wedge \tau_n}\rvert^{5p})\nonumber \\
   &+\beta^{2p}(1+V(X_{\kappa_n(s)\wedge \tau_n})^{2p})\lvert W_{s\wedge  \tau_n}-W_{\kappa_n(s)\wedge \tau_n}\rvert^{2p},
\end{align}
so that one obtains by the independence of $X_{\kappa_n(s)\wedge \tau_n}$ with $W_{s\wedge  \tau_n}-W_{\kappa_n(s)\wedge \tau_n}$ that
\begin{align}\label{eq: increment bound V hat}
   &(1+\hat{V}(X_{s\wedge  \tau_n})^p+\hat{V}(X_{\kappa_n(s)\wedge  \tau_n})^p)\lvert X_{s\wedge  \tau_n}-X_{\kappa_n(s)\wedge \tau_n}\rvert^{2p} \nonumber \\
   &\leq H_{2p}(1+\sup_{s\in [0,T]}E\hat{V}(X_{s\wedge  \tau_n})^{5p}+EV(X_{\kappa_n(s)\wedge \tau_n})^{5p}+E\lvert X_{\kappa_n(s)\wedge \tau_n}\rvert^{5p})n^{-2p}\nonumber \\
   &+H_{2p}(1+EV(X_{\kappa_n(s)\wedge \tau_n})^{2p})n^{-p} \nonumber \\
   &\leq H_{2p}(1+\sup_{s\in [0,T]}E\hat{V}(X_{s\wedge  \tau_n})^{5p}+EV(X_{\kappa_n(s)\wedge \tau_n})^{5p}+E\lvert X_{\kappa_n(s)\wedge \tau_n}\rvert^{5p})n^{-p}.
\end{align}
For $I^{13}_t$ the diffusion part of the increment term vanishes since the other factors are $\mathcal{F}_{\kappa_n(s)}$-measurable, so that additionally applying the bound in \eqref{eq: nabla b bound} one obtains
\begin{align}
   EI^{13}_t&\leq c E\int^{t\wedge \tau_n}_0 \lvert \nabla b(X_{\kappa_n(s)}) \rvert  \lvert e_{\kappa_n(s)}\rvert^{p-1} \biggr \lvert \int^s_{\kappa_n(s)}b(X_u) du \biggr \rvert ds\nonumber \\
   &\leq c E\int^t_0 \lvert \nabla b(X_{\kappa_n(s)}) \rvert  \lvert e_{\kappa_n(s)}\rvert^{p-1} \biggr \lvert \int^s_{\kappa_n(s)}b(X_u) du \biggr \rvert ds\nonumber \\
   &\leq c \int^t_0 \biggr (E\lvert e_{\kappa_n(s)}\rvert^p+E V(X_{\kappa_n(s)})  ^{p} E\biggr \lvert \int^s_{\kappa_n(s)}b(X_u) du \biggr \rvert^{p}\biggr) ds.
\end{align}
Now one may apply Hölder's inequality and \eqref{eq: upper bound for b 2} to obtain
\begin{align}
   EI^{13}_t&\leq c \int^t_0 \biggr (E\lvert e_{\kappa_n(s)}\rvert^p+n^{1-p}E V(X_{\kappa_n(s)})  ^{p}   \int^s_{\kappa_n(s)}\lvert b(X_u) \rvert^{p}du  \biggr) ds\nonumber \\
   &\leq c \int^t_0 E\lvert e_{\kappa_n(s)}\rvert^pds+\int^t_0n^{1-p}E V(X_{\kappa_n(s)})  ^{p}   \int^s_{\kappa_n(s)}H_{p}(1+V(X_u)^{2p}+\lvert X_u\rvert^{2p})du   ds\nonumber \\
   &\leq c \int^t_0 E\lvert e_{\kappa_n(s)}\rvert^pds+\int^t_0n^{1-p}E   \int^s_{\kappa_n(s)}H_{p}(1+V(X_{\kappa_n(s)})^{3p}+V(X_u)^{3p}+\lvert X_u\rvert^{3p})du   ds.
\end{align}
Finally for $I^{14}_t$ one first observes
\begin{equation}
    \lvert x \rvert^{p-2}x-\lvert y \rvert^{p-2}y=\lvert x \rvert^{p-2}(x-y)+y(\lvert x \rvert^{p-2}-\lvert y \rvert^{p-2}),
\end{equation}
so that applying the mean value theorem to $\lvert \cdot \rvert^{p-2}$ there exists $t\in [0,1]$
\begin{align}
    \lvert x \rvert^{p-2}x-\lvert y \rvert^{p-2}y&\leq \lvert x \rvert^{p-2}( x-y)+c \lvert y\rvert \lvert tx+(1-t)y\rvert^{p-4}\langle tx+(1-t)y,x-y\rangle,\nonumber \\
    &\leq c( \lvert x \rvert^{p-2}+\lvert y \rvert ^{p-2})\lvert x-y\rvert.
\end{align}
Therefore one has
\begin{align}
   & I^{14}_t\leq c\int^{t\wedge \tau_n}_0 (1+V(X_{\kappa_n(s)})+V(X^n_{\kappa_n(s)}))\lvert e_{\kappa_n(s)} \rvert\nonumber \\
   &\times( \lvert e_s \rvert^{p-2}+\lvert e_{\kappa_n(s)} \rvert ^{p-2})\lvert e_s-e_{\kappa_n(s)}\rvert  ds \nonumber \\
   &\leq c\int^{t\wedge \tau_n}_0 ( \lvert e_s \rvert^p+\lvert e_{\kappa_n(s)} \rvert ^p) ds\nonumber \\
   &+c\int^{t\wedge \tau_n}_0 (1+V(X_{\kappa_n(s)\wedge \tau_n})+V(X^n_{\kappa_n(s)}))^p \lvert e_{s\wedge \tau_n}-e_{\kappa_n(s)\wedge \tau_n}\rvert^pds.
\end{align}
Furthermore using Assumption \ref{assmp: Lyapunov fn}
\begin{align}
&\lvert e_{s\wedge \tau_n}-e_{\kappa_n(s)\wedge \tau_n}\rvert^p \leq cn^{1-p}\int^{s\wedge \tau_n}_{\kappa_n(s)\wedge \tau_n}\lvert b(X_u)-b(X^n_{\kappa_n(u)})\rvert^p du\nonumber \\
&\leq cn^{1-p}\int^{s\wedge \tau_n}_{\kappa_n(s)\wedge \tau_n}c(1+V(X_u)^p+V(X^n_{\kappa_n(u)})^p)(\lvert X_u\rvert^p+\lvert X^n_{\kappa_n(u)}\rvert^p )ds\nonumber \\
&\leq cn^{1-p}\int^{s}_{\kappa_n(s)}c(1+V(X_{u})^{2p}+V(X^n_{\kappa_n(u)\wedge \tau_n})^{2p}+\lvert X_u\rvert^{2p}+\lvert X^n_{\kappa_n(u)\wedge \tau_n}\rvert^{2p} )ds.
\end{align}
Therefore
\begin{align}
 &E\int^{t\wedge \tau_n}_0 (1+V(X_{\kappa_n(s)})+V(X^n_{\kappa_n(s)}))^p \lvert e_{s\wedge \tau_n}-e_{\kappa_n(s)\wedge \tau_n}\rvert^pds   \nonumber \\
 &\leq c\int^t_0n^{1-p}\int^{s}_{\kappa_n(s)}c(1+EV(X_{u})^{3p}+EV(X_{\kappa_n(u)})^{3p}+EV(X^n_{\kappa_n(u)\wedge \tau_n})^{3p}+\nonumber \\
 &+E\lvert X_u\rvert^{3p}+E\lvert X^n_{\kappa_n(u)\wedge \tau_n}\rvert^{3p} )duds.
\end{align}
Summing the previous terms one sees that the error is proportional to $n^{-p}$ as required. In order to access the dependence on the two generic constants, note that by Proposition \ref{prop: mom bounds} the highest order dependence must come from \eqref{eq: increment bound V hat}, and in particular one obtains via Young's inequality and the properties \eqref{eq: constant condition} that
\begin{align}
    \sup_{u\in [0,t]}E\lvert e_{u\wedge \tau_n}\rvert &\leq \int^t_0 (H_{2p}H_{10p+q_{5p}'}+H_{2p}\hat{H}_{5p+q'_{5p}}+\sup_{u\in[0,s]}E\lvert e_{u\wedge \tau_n}\rvert^p)ds \nonumber \\
    &\leq \int^t_0 (H_{12p+q_{5p}'}+\hat{H}_{7p+q'_{5p}}+\sup_{u\in[0,s]}E\lvert e_{u\wedge \tau_n}\rvert^p)ds,
\end{align}
and so by applying Grönwall's inequality one obtains
\begin{equation}
    \sup_{u\in [0,T]}E\lvert e_{u\wedge \tau_n} \rvert\leq (H_{12p+q_{5p}'}+\hat{H}_{7p+q'_{5p}})n^{-p}.
\end{equation}
For $I^2_t$ and $I^3_t$ we perform a refinement of the argument in the previous section. In particular, applying Lemma \ref{lemma: decay of stopping time} once more one has
    \begin{align}
        EI^2_t&=c(E\lvert X^n_{s\wedge \tau_n}-X^n_s\rvert^{2p})^{1/2}P(
        \tau_n\leq T)^{1/2}\nonumber \\
        &\leq \sqrt{H_{12p+q'_{6p}}H_{24p+q'_{12p}}n^{-2p}}\nonumber \\
        &\leq H_{24p+q'_{12p}}n^{-p}.
    \end{align}
    As before, a similar argument holds for $I^3_t$. Collating all relevant terms together the result follows by the properties of the generic constant outlined in \eqref{eq: constant condition}.

\section{Singular Interacting Particles}\label{sec: singular ips}
Now we move on to our main application. Following \cite{Guillin2022OnSO}, we consider an interacting particle system of the form
\begin{equation}\label{eq: key IPS}
    dX^{i,N}_t=-\frac{1}{N}\sum_{j=1, \; j\not=i}^N U'(X^{i,N}_t-X^{j,N}_t)dt-Q'(X^{i,N}_t)dt+\sqrt{2\sigma/N}dW^i_t,
\end{equation}
for $i=1,2,...,N$, with initial conditions satisfying $X^{1,N}_0<X^{2,N}_0<...<X^{N,N}_0$. Since we consider one-dimensional particles, the full system $X^N_t=(X^{1,N}_t,...,X^{N,N}_t)$ is considered as a vector in $\mathbb{R}^N$. Additionally, $\sigma>0$ is a parameter that controls the noise, and one has the potentials $U\in C^1(\mathbb{R}\setminus \{0\})$ and $Q\in C^1(\mathbb{R})$, usually referred to as interaction and confinement potentials respectively. We assume that $U$ and $Q$ take a particular specific form (mostly to simplify calculations)
\begin{assmp}\label{assmp: specific assmp}
    Suppose for $\alpha> 1$ one has that $U:\mathbb{R}\to\mathbb{R}$ is of the form
    \begin{equation}
U(x):=\frac{1}{\alpha-1}\lvert x\rvert ^{-\alpha+1},
    \end{equation}
    so that 
    \begin{equation}
        U'(x):=-\frac{x}{\lvert x \rvert^{\alpha+1}}.
    \end{equation}
    Assume additionally that $Q\in C^1(\mathbb{R})$ is such that $Q'$ is Lipschitz.
\end{assmp}
Under this assumption each pair of particles a distance $r>0$ away from one another repulses each another (or in other words contributes to the others instantaneous drift) with magnitude $r^{-\alpha}$. Therefore \eqref{eq: key IPS} becomes
\begin{align}\label{eq: key IPS 2}
    dX^{i,N}_t&=\frac{1}{N}\biggr ( \sum_{j=1,2,...,i-1}( X^{i,N}_t-X^{j,N}_t)^{-\alpha}-\sum_{j=i+1,i+2,...,N}( X^{j,N}_t-X^{i,N}_t)^{-\alpha}\biggr)dt\nonumber \\
    &+Q'(X^{i,N}_t)dt+\sqrt{2\sigma/N}dW^i_t,
\end{align}
and since this enforces that $U(-x)=U(x)$, for
\begin{equation}\label{eq: U_N defn}
    U_N(x):=\frac{1}{N}\sum_{1\leq i<j\leq N} U(x^i-x^j),\qquad Q_N(x):=\frac{1}{N}\sum_{k=1}^NQ(x^k),
\end{equation}
one may furthermore rewrite \eqref{eq: key IPS} as
\begin{equation}\label{eq: key IPS grad form}
    dX^N_t=-\nabla U_N(X^N_t)dt-\nabla Q_N(X^N_t)+\sqrt{2\sigma/N}dW_t.
\end{equation}

\subsection{Connection to General Framework}\label{subsec: connection general framework}
    In order to prove the convergence of an Euler scheme method for \eqref{eq: key IPS}, it shall be necessary to show that the framework of Assumption \ref{assmp: Lyapunov fn} and Assumption \ref{assmp: Lyapunov controls nabla b} applies to this setting. To this end
    \begin{itemize}
        \item we let the open set for which \eqref{eq: key IPS} takes values be given as
        \begin{equation}\label{eq: D_N defn}
    D_N:=\{x\in \mathbb{R}^N\vert x_1< x_2<...<x_N\}\subset \mathbb{R}^N,
\end{equation}
i.e. the set which preserves the ordering of the initial conditions.
\item we set $b=-\nabla U_N-\nabla Q_N$ and $\beta=\sqrt{2\sigma/N}$
\item the Lyapunov function we use shall be given as 
\begin{equation}\label{eq: V_N defn}
        V_N(x):=c+N^{\frac{2}{\alpha-1}}U_N(x)^{1+\frac{2}{\alpha-1}}= c+\frac{c}{N}\biggr (\sum_{1\leq i<j\leq N} \lvert x^i-x^j\rvert^{1-\alpha} \biggr)^{\frac{\alpha+1}{\alpha-1}},
    \end{equation}
     \begin{equation}\label{eq: V_N hat defn}
        \hat{V}_N(x):=c+N^{\frac{3}{\alpha-1}}U_N(x)^{1+\frac{3}{\alpha-1}}= c+\frac{c}{N}\biggr (\sum_{1\leq i<j\leq N} \lvert x^i-x^j\rvert^{1-\alpha} \biggr)^{\frac{\alpha+2}{\alpha-1}},
    \end{equation}
    where $c>0$ depends only on $\alpha>0$ and the Lipschitz constant of $\nabla Q_N$
    \item as a result, the tamed Euler scheme approximation $\hat{X}^{n,N}_t$ following the framework established earlier is given as
    \begin{equation}\label{eq: IPS Euler scheme}
      d\bar{X}^{n,N}_t =b_{n.N}(\bar{X}^{n,N}_{\kappa_n(t)})dt+\sqrt{2\sigma/N}dW_t,
    \end{equation}
    for 
    \begin{align}\label{eq: tamed IPS coeff}
        b_{n,N}^i(x)&:=\biggr (-\nabla U_N(x)-\nabla Q_N(x)\biggr)\nonumber \\
        &\times1_{x\in D_N}\cdot1_{V_N(x)\leq n^{1/2}}\cdot 1_{\lvert \nabla U_N(x)+\nabla Q_N(x)\rvert\leq n^{1/2}}.
    \end{align}
    \end{itemize}

\subsection{Main Results}
Let us define
 \begin{defn}
        Let $\textit{poly}(N)$ denote any quantity of the form $\textit{poly}(N)=cN^{l}$ for some $c,l>0$ that do not depend on $N\in \mathbb{N}$ or the stepsize parameter $n\geq 1$.
    \end{defn}   
    Our main result of this section is then
\begin{thm}\label{thm: conv of particle Euler scheme}
    Consider \eqref{eq: key IPS}, and suppose Assumption \ref{assmp: specific assmp} holds. Suppose there exists $c,p_*>0$ such that for every $N\in \mathbb{N}$ the sequence of initial conditions $x_{0,N}\in D_N$ satisfies
    \begin{equation}
        \max_{i=2,...,N}\lvert x_{0,N}^i-x^{i-1}_{0,N}\rvert^{-1}=\textit{poly}(N),\;\;\;\lvert x_{0,N}\rvert =poly(N).
    \end{equation}
    Then for every $p,T>0$, $N\in \mathbb{N}$ and $n\in \mathbb{N}$ satisfying 
    \begin{equation}
    n\geq \max\{V_N(x)^2, \hat{V}_N(x)^2,(\nabla U(x)+\nabla Q(x))^2\}
    \end{equation}
    one has for \eqref{eq: IPS Euler scheme} and \eqref{eq: tamed IPS coeff} that
     \begin{equation}
       \sup_{t\in [0,T]}E\lvert \bar{X}^{n,N}_t-X^N_t\rvert^p\leq \textit{poly}(N)n^{-p}.
    \end{equation}
\end{thm}

It is clear that this result follows directly from Theorem \ref{thm: conv of scheme 2} after verifying Assumption \ref{assmp: Lyapunov fn} and Assumption \ref{assmp: Lyapunov controls nabla b} for the choice of SDE and Lyapunov functions given above.

\subsection{Verifying Lyapunov Properties}
\begin{note}
    In this section the generic constant $c>0$ does \textbf{not} depend on $p>0$, but does depend on the parameter $\alpha>1$. This is so that certain properties of the constants given in Assumption \ref{assmp: Lyapunov fn} and \ref{assmp: Lyapunov controls nabla b} can be verified.
\end{note}
Let us first show that $V_N$ and $\hat{V}_N$ fit the Lipschitz-type assumptions in \eqref{eq: Lyapunov Lipschitz b} and \eqref{eq: higher order Lyapunov}.
\begin{lemma}\label{lem: Lyapunov lipschitz}
    There exists a constant $c>0$ not depending on $N\in \mathbb{N}$ such that for every $x,y \in D_N$ one has
   \begin{equation}
       \lvert \nabla U_N(x)-\nabla U_N(y)\rvert\leq (V_N(x)+V_N(y)) \lvert x-y\rvert,
   \end{equation}
   for $V_N$ as given in \eqref{eq: V_N defn}.
\end{lemma}
\begin{proof}
     First let $x,y \in D_N, x\not=y$ be arbitrary. First recall for $i=1,2,...,N$ one has
     \begin{equation}
\partial_i U_N(x)=\frac{1}{N}\biggr (\sum_{j=i+1,i+2,...,N}( x_j-x_i)^{-\alpha}-\sum_{j=1,2,...,i-1}( x_i-x_j)^{-\alpha}\biggr),
     \end{equation}
so
\begin{align}
    \lvert \nabla U_N(x)-\nabla U_N(y)\rvert&\leq \sum_{i=1}^N\lvert  \partial_i U_N(x)-\partial_i U_N(y) \rvert\nonumber \\
    &\leq \frac{c}{N}\sum_{1\leq i<j\leq N}\lvert (x_i-x_j)^{-\alpha}-(y_i-y_j)^{-\alpha}\rvert.
\end{align}
 Now observe that by the mean value theorem, for each $1\leq j<i\leq N$ there exists $z_{ij}\in \mathbb{R}$ lying inbetween $x_i-x_j$ and $y_i-y_j$ such that
 \begin{align}
 \lvert  \partial_i U_N(x)-\partial_i U_N(y) \rvert &\leq \frac{c}{N}\sum_{1\leq i<j\leq N}\lvert z_{ij}^{-\alpha-1}((x_i-x_j)-(y_i-y_j))\rvert\nonumber \\
  &\leq \lvert x-y\rvert\frac{c}{N}\sum_{1\leq i<j\leq N} z_{ij}^{-\alpha-1}\nonumber \\
  &\leq \lvert x-y\rvert\frac{c}{N}\sum_{1\leq i<j\leq N} ( ( x_j-x_i)^{-\alpha-1}+( y_j-y_i)^{-\alpha-1}).
 \end{align}
 The last bound follows from the monotonic nature of the function $x\mapsto x^{-\alpha-1}$ for $x>0$. Then the result follows since for every $x_1,...,x_m>0$ and $w>1$ one has
 \begin{equation}\label{eq: power expression}
  \sum_{i=1}^mx_i^w \leq  \biggr ( \sum_{i=1}^mx_i \biggr)^w .
 \end{equation}
\end{proof}
\begin{lemma}\label{lem: lyapunov lipschitz interacting particles}
    There exists a constant $c>0$ not depending on $N\in \mathbb{N}$ such that for every $x,y \in D_N$ one has
    \begin{equation}
       \lvert \nabla U_N(x)-\nabla U_N(y)-\nabla^2 U_N(x)(x-y)\rvert\leq (\hat{V}_N(x)+\hat{V}_N(y))\lvert x-y\rvert,
   \end{equation}
   for $\hat{V}_N$ as given in \eqref{eq: V_N hat defn}. 
\end{lemma}
\begin{proof}
 Following the line of argument from the proof above one writes
   \begin{align}
       &\lvert \nabla U_N(x)-\nabla U_N(y)-\nabla^2 U_N(x)(x-y)\rvert\nonumber \\
       &\leq \frac{c}{N}\sum_{1\leq j<i\leq N}\lvert (x_i-x_j)^{-\alpha}-(y_i-y_j)^{-\alpha}+\alpha (x_i-x_j)^{-\alpha-1}((x_i-x_j)-(y_i-y_j))\rvert.\nonumber
   \end{align}
   Once more we apply the mean value theorem to find $\tilde{z}_{ij}$ between $x_i-x_j$ and $y_i-y_j$ such that 
     \begin{align}
       &\lvert \nabla U_N(x)-\nabla U_N(y)-\nabla^2 U_N(x)(x-y)\rvert\nonumber \\
       &\leq \frac{c}{N}\sum_{1\leq j<i\leq N}\tilde{z}_{ij}^{-\alpha-2}((x_i-x_j)-(y_i-y_j))^2\nonumber \\
       &\leq \lvert x-y\rvert^2\frac{c}{N}\sum_{1\leq j<i\leq N}\tilde{z}_{ij}^{-\alpha-2}\nonumber \\
       &\leq \lvert x-y\rvert^2\frac{c}{N}\sum_{1\leq j<i\leq N}((x_i-x_j)^{-\alpha-2}+(y_i-y_j)^{-\alpha-2}),
   \end{align}
   at which point the result follows from applying \eqref{eq: power expression} once more.
\end{proof}
Now we prove bounds that shall allow us to prove the Lyapunov condition \eqref{eq: generator assmp}. To this end let us define
\begin{equation}\label{eq: l defn}
        l(x):=\min\{\min_{j=2,...,N} ( x^j-x^{j-1}),1\}.
    \end{equation}
    The additional minimum here is to enforce that $l(x)\leq 1$.
\begin{lemma}\label{lem: bounds in terms of min}
    Assume Assumption \ref{assmp: specific assmp} and consider $U_N$ as given in \eqref{eq: U_N defn}.  Then for $x\in D$ one has
\begin{equation}\label{eq: bounds on U_N}
      cN^{-1}l(x)^{1-\alpha} \leq U_N(x)\leq cNl(x)^{1-\alpha},
    \end{equation}
    \begin{equation}\label{eq: bounds on nabla U_N}
      cN^{-3/2}l(x)^{-\alpha} \leq \lvert \nabla U_N(x)\rvert \leq  cl(x)^{-\alpha},
    \end{equation}
    \begin{equation}\label{eq: bounds on Delta U_N}
      \lvert \Delta U_N(x)\rvert \leq  cl(x)^{-\alpha-1},
    \end{equation}
    \begin{equation}
        \langle \nabla U_N, \nabla Q_N\rangle\leq cl(x)^{1-\alpha}.
    \end{equation}
\end{lemma}
\begin{proof}
    Let $x\in D$. We begin by showing bounds on $U_N$. For the upper bound observe that for $j>i$ one has 
    \begin{equation}\label{eq: difference particles}
        x^j-x^i\geq (j-i)l(x),
    \end{equation}
    and therefore
    \begin{align}
    U_N(x)&=\frac{1}{(\alpha-1)N}\sum_{1\leq i<j\leq N}(x^j-x^i)^{-\alpha+1}\nonumber \\
    &\leq \frac{l(x)^{-\alpha+1}}{(\alpha-1)N}\sum_{1\leq i<j\leq N}(j-i)^{-\alpha+1}.
\end{align}
Then bounding each term in the sum by $1$, since there are $cN^2$ terms, the upper bound follows. Now for our fixed $x\in D$ let $i\in \{2,...,N\}$ be the index such that
     \begin{equation}
   x_i-x_{i-1}=l(x).
    \end{equation}
    Then the lower bound on $U_N$ follows by considering only the term in the sum that achieves this minimum.

    For $\nabla U_N$, let be $i$ as given above, and consider the vector $\xi\in \mathbb{R}^N$ such that $\xi_k=-1$ for $k\geq i$ and $\xi_k=1$ for $k<i$. Then the product $  \langle \nabla U(x), \xi\rangle$ will only contain terms $(x_{k_1}-x_{k_2})^{-\alpha}$ for indices $k_2\leq i\leq k_1$, $k_1\not=k_2$. As a result
    \begin{equation}
        \langle \nabla U_N(x), \xi\rangle \geq \frac{1}{N}(x_i-x_{i-1})^{-\alpha}=\frac{l(x)^{-\alpha}}{N}.
    \end{equation}
    Then since $\lvert \xi\rvert=N^{1/2}$, by Cauchy-Schwarz the lower bound on $\lvert \nabla U_N(x)\rvert$ follows. For the upper bound observe that by \eqref{eq: difference particles}
   \begin{equation}
       \lvert \nabla U_N(x)\rvert \leq N^{-1}\sum_{1\leq j<i\leq N}(x_i-x_j)^{-\alpha-1}\leq c\frac{l(x)^{-\alpha}}{N}\sum_{i=1}^N\sum_{j=i+1}^N(j-i)^{-\alpha},
   \end{equation}
   and so the result follows since $\sum_{i=1}^\infty i^{-\alpha}\leq c$. The bound for $\Delta U_N$ follows in the same way. Finally
   \begin{align}
       \langle \nabla U_N, \nabla Q_N\rangle &\leq c N^{-2}\sum_{1\leq j<i\leq N}( x^i-x^j)^{-\alpha}\lvert \nabla Q(x^i)-\nabla Q(x^j)\rvert \nonumber \\
       &\leq c N^{-2}\sum_{1\leq j<i\leq N}( x^i-x^j)^{1-\alpha},
   \end{align}
   so one obtains the result from \eqref{eq: difference particles}.
\end{proof}
To show the actual Lyapunov condition \eqref{eq: generator assmp} holds we shall need the following technical lemma
\begin{lemma}\label{lem: maximum of poly}
    Let $p>q>0$ and $a,b>0$. Then there exists a constant $c=c(p,q)>0$ such that
    \begin{equation}
        \sup_{t>0}(-at^{-p}+bt^{-q})\leq b\biggr (\frac{bq}{ap}\biggr)^{\frac{q}{p-q}}\,.
    \end{equation}
\end{lemma}
\begin{proof}
    By single variable calculus one sees that $f(t)=-at^{-p}+bt^{-q}$ has a positive minimum point at
    \begin{equation}
        t=\biggr(\frac{ap}{bq}\biggr)^{\frac{1}{p-q}}.
    \end{equation}
    The result then follows from substituting this extreme into $f$ and ignoring the first (negative) term.
\end{proof}
\begin{lemma}\label{lem: assump for interacting particles}
    Assume Assumption \ref{assmp: specific assmp}. Then for every $p>0$ there exist $q_p,\hat{q}_p>1$ which are decreasing in $p>0$ and independent of $N\in \mathbb{N}$, such that
    \begin{equation}
     \sup_{x\in D_N} ( \mathcal{L} V_N^p)(x)+\frac{1}{q_p}\lvert \nabla (V_N^p)(x)\rvert^{q_p}\leq a_p,
\end{equation}
    \begin{equation}
     \sup_{x\in D_N} ( \mathcal{L} \hat{V}_N^p)(x)+\frac{1}{\hat{q}_p}\lvert \nabla (\hat{V}_N^p)(x)\rvert^{\hat{q}_p}\leq \hat{a}_p.
\end{equation}
Furthermore these constants satisfy $a_p,\hat{a}_p=poly(N)$ and $(a_q)^{p/q}\leq a_p$, $(\hat{a}_q)^{p/q}\leq \hat{a}_p$ for all $0<q<p$.
\end{lemma}
\begin{proof}
Recall that for the proof of this lemma the generic constant $c>0$ shall \textbf{not} depend on $p>0$, so as to show the final condition on the constants $a_p>0$ given in the lemma. Let $\mathcal{L}_N$ be the generator of \eqref{eq: key IPS}. Then one calculates for $U\in C^2(D_N)$ that
\begin{align}
    \mathcal{L}_N U^p&=-pU^{p-1}\lvert \nabla U\rvert^2+pU^{p-1}\langle \nabla U, \nabla Q_N\rangle\nonumber \\
    &+ p U^{p-1}\Delta U+p(p-1)U^{p-2}\lvert \nabla U\rvert^2 .
\end{align}
Using Lemma \ref{lem: bounds in terms of min} then yields that for $p\geq 2$
\begin{equation}
 -pU_N^{p-1}\lvert \nabla U_N\rvert^2\leq  -cpN^{-p-2}l^{(p-1)(1-\alpha)-2\alpha},
\end{equation}
\begin{equation}
pU_N^{p-1}\langle \nabla U_N, \nabla Q_N\rangle\leq cpN^{p-1}l^{(p-1)(1-\alpha)+1-\alpha},
\end{equation}
\begin{equation}
    p U_N^{p-1}\Delta U_N\leq cpN^{p-1}l^{(p-1)(1-\alpha)-1-\alpha},
\end{equation}
\begin{equation}
p(p-1)U_N^{p-2}\lvert \nabla U_N\rvert^2\leq cp(p-1)N^pl^{(p-1)(1-\alpha)-1-\alpha}.
\end{equation}
Therefore, since $0\leq l\leq 1$ by it's definition in \eqref{eq: l defn}, one may bound $l^{-r_1}\leq l^{-r_2}$ for $r_1>r_2$. Additionally, since $\alpha>1$ one obtains
\begin{align}
    \mathcal{L}_N U_N^p&\leq -cpN^{-p-2}l^{(p-1)(1-\alpha)-2\alpha}+cp(p-1)N^pl^{(p-1)(1-\alpha)-1-\alpha} .
\end{align}
Furthermore, one has
\begin{equation}
    \lvert \nabla (U_N)^p\rvert^q=p\lvert U_N\rvert^{q(p-1)}\lvert \nabla U_N\rvert^q\leq cpN^{q(p-1)}l^{q[(p-1)(1-\alpha)-\alpha]}.
\end{equation}
Now let us fix $r,s>0$ and write $\bar{V}_N:=N^rU_N^s$. Then letting $p\geq 2/s$ and setting 
\begin{equation}\label{eq: q defn}
    q=\frac{(ps-1)(\alpha-1)+1+\alpha}{(ps-1)(\alpha-1)+\alpha},
\end{equation}
one obtains
\begin{align}
   & \mathcal{L}_N \bar{V}_N^p+ \frac{1}{q}\lvert \nabla (\bar{V}_N)^p\rvert^q\leq -cpsN^{rp-p-2}l^{(ps-1)(1-\alpha)-2\alpha}\nonumber \\
    &+ cps(ps-1)N^{rpq+q(ps-1)}l^{(ps-1)(1-\alpha)-\alpha-1}.
\end{align}
This is clearly bounded by Lemma \ref{lem: maximum of poly}, and one may derive the relevant expressions for $V_N$ and $\hat{V}_N$ by substituting in firstly $r=\frac{2}{\alpha-1}$, $s=1+\frac{2}{\alpha-1}$ and secondly $r=\frac{3}{\alpha-1}$, $s=1+\frac{3}{\alpha-1}$, along with the corresponding values of $q_p>1$ and $\hat{q}_p>1$ coming from \eqref{eq: q defn}. Furthermore, the resulting dependence on $N\in \mathbb{N}$ is clearly polynomial. To check that the final condition holds, one notes that the dependence on $N\in \mathbb{N}$ is super-polynomial in $N\in \mathbb{N}$, so that such a domination condition must hold.
\end{proof}

\subsection{Proof of Theorem \ref{thm: conv of particle Euler scheme}}
The result now follows simply by combining the results of the proceeding section and checking the constants have only polynomial dependence on $N\in \mathbb{N}$. To this end we merely need to check Assumption \ref{assmp: Lyapunov fn} and Assumption \ref{assmp: Lyapunov controls nabla b} hold for each $N\in \mathbb{N}$ and apply Theorem \ref{thm: conv of scheme 2}. Let us being with Assumption \ref{assmp: Lyapunov fn}. 
\begin{itemize}
    \item the set $D_N\subset \mathbb{R}$, drift coefficient $b:D_N\to\mathbb{R}^N$ and function $V_N:D_N\to [0,\infty)$ are given in Section \ref{subsec: connection general framework}
    \item the fact \eqref{eq: Lyapunov Lipschitz b} holds for $V_N$ and $b_N$ is the content of Lemma \ref{lem: Lyapunov lipschitz}
    \item the fact $V_N,\hat{V}_N:D_N\to [0,\infty)$ obey a condition of the form \eqref{eq: generator assmp}, and that the constants scale polynomially with $N\in \mathbb{N}$, is the content of Lemma \ref{lem: assump for interacting particles}. This particularly implies $b_p=\hat{b}_p=0$ for all $p\geq 2$ in this setting
    \item one sees that for every sequence $x_n\in D_N$ such that $x_n\to x\in \partial D_N$ one has $V_N(x_n),\hat{V}_N(x_n)\to \infty$, since this implies $(x_n)_i\to (x_n)_{i+1}$ for some $i=2,...,N$
    \item the condition \eqref{eq: nabla V power bound} follows from Lemma \ref{lem: bounds in terms of min} and the definition of $V_N$ and $\hat{V}_N$
\end{itemize}
It remains to prove the monotonicity condition \eqref{eq: mon assump} on the full drift coefficient, for a constant $\mu>0$ \textbf{independent} of $N\in \mathbb{N}$. Since $Q'$ is Lipschitz one notes that it is sufficient to show for $u_N^{ij}(x):=\lvert x^i-x^j\rvert^{1-\alpha}$
    \begin{equation}
        \langle \nabla u^{ij}_N(x)-\nabla u^{ij}_N(y),x-y\rangle\leq 0.
    \end{equation}
    Indeed, one has
    \begin{equation}
        \langle \nabla u^{ij}_N(x)-\nabla u^{ij}_N(y),x-y\rangle = (1-\alpha)(\lvert x^i-x^j\rvert^{-\alpha}-\lvert y^i-y^j\rvert^{-\alpha})((x^i-x^j)-(y^i-y^j)).
    \end{equation}
    The result then follows since $x\mapsto x^{-\alpha}$ is decreasing for $x>0$. It is easy to verify that, given the assumptions on $x_{0,N}$, the generic constants $H_p$ and $\hat{H}_p$ all have polynomial dependence on $N\in \mathbb{N}$. Therefore the result follows.

    \section{Recovering Tamed Schemes}\label{sec: recovering tamed schemes}
Let us consider a context common in the tamed scheme literature
\begin{assmp}\label{assmp: tamed schemes}
    Suppose $b:\mathbb{R}^d\to\mathbb{R}^d$ is such that there exists $l>0$ for which \eqref{eq: tamed assmp} holds, and furthermore $b\in C^1(\mathbb{R}^d)$ and one has for every $x,y\in \mathbb{R}^d$ that
  \begin{equation}
        \langle  b(x)- b(x) -\nabla b(x)( x-y)\leq c(1+\lvert x \rvert^{l+1}+\lvert y\rvert^{l+1})\lvert x -y\rvert^2.
    \end{equation}  
    \begin{equation}
        \langle b(x)-b(y), x-y\rangle\leq c\lvert x -y\rvert^2.
    \end{equation}
\end{assmp}
Then we have the following result which essentially recovers (with a different scheme) the result of \cite{articletamed}.
\begin{thm}\label{thm: recovering tamed scheme}
    Let $V(x):=c(1+\lvert x \rvert^l)$, $D=\mathbb{R}^d$ and consider the resulting tamed scheme given by \eqref{eq: truncated coef} and \eqref{eq: scheme defn}. Then for every $T>0$ and $p\geq 2$ there exists $c_*>0$ independent of $n\geq 1$ such that
    \begin{equation}
        E\sup_{t\in [0,T]}\lvert X^n_t-X_t\rvert^p \leq c_*n^{-p}.
    \end{equation}
\end{thm}
\begin{proof}
    It suffices to show that Assumptions \ref{assmp: Lyapunov fn} and \ref{assmp: Lyapunov controls nabla b} hold. By the monotonicity assumption it follows that
    \begin{equation}
\langle b(x), x\rangle\leq c(1+\lvert x \rvert^2),
\end{equation}
from which it follows that \eqref{eq: generator assmp} holds for every $p>0$ some $q_p,a_p,b_p>0$. Furthermore, we can choose these coefficients to grow fast enough to satisfy the monotonicity requirements. The other conditions hold trivially given the assumptions. Furthermore, Assumption \ref{assmp: Lyapunov controls nabla b} holds choosing $\hat{V}(x):=c(1+\lvert x \rvert^{l+1})$.
\end{proof}

\appendix
\section{Itô's Formula}

We begin with a statement of Itô's formula for a function defined on an open set, that could potentially blow up on the boundary of this set. This result is essentially a restatement of Proposition 1 in \cite{JOHNSTON2026104772}.
\begin{prop}\label{prop: local Ito}
    Let $(W_t)_{t\geq 0}$ be a Wiener martingale defined on a filtered probability space $(\Omega, P, \mathcal{F}, (\mathcal{F}_t)_{t\geq 0})$. Let $(Y_t)_{t\geq 0}$ be a $\mathcal{F}_t$-measurable process satisfying
    \begin{equation}
        Y_t=Y_0+\int^t_0 g_s ds+\beta W_t,
    \end{equation}
    where $\beta>0$ is a constant, and $g_t$ is a $\mathcal{F}_t$-measurable processes satisfying  
    \begin{equation}
\int^t_0 g_s ds<\infty,
    \end{equation}
    almost surely. Let $D\subset \mathbb{R}^d$ be a non-empty open set, $\tau$ a stopping time such that $X_{t\wedge  \tau}\in D$ for all $t>0$ almost surely, and $f\in C^2(D)$. Then one recovers the standard Itô's formula for $f(Y_{t\wedge\tau})$, that is 
    \begin{equation}
        f(Y_{t\wedge\tau})= f(y_0)+\int^{t\wedge\tau}_0 (\langle \nabla f(Y_s), g_s\rangle +\frac{\beta}{2}\Delta f(Y_s))ds+\int^{t\wedge\tau}_0 \beta \langle \nabla f(Y_s), dW_s\rangle.
    \end{equation}
\end{prop}
\begin{proof}
Let us find a set where $f$ and it's first two derivatives are bounded, and furthermore this property always holds on a small ball around each point. Specifically, let us define
\begin{equation}
    D_a:=\{x\in D\vert \; \lvert f(y)\rvert< a,\;\; \lvert \nabla f(y)\rvert <a,\;\; \lvert \Delta f(y)\rvert< a \; \text{for} \;y\in B_{1/a}(x)\}.
\end{equation}
Since $f\in C^2(D)$ and $D$ is non-empty, this set must be non-empty for some $a>0$. Furthermore, let us define the first exist from this set as the stopping time
\begin{equation}
    \tau_a:=\inf\{t\geq 0\vert \; X_t\not\in D_a\}.
\end{equation}
Then we may consider a mollification of $f$ restricted to $D_a$. Particularly, let us fix $\phi\in C^\infty_c(\mathbb{R}^d)$ satisfying $\lvert \lvert \phi\rvert\rvert_1=1$, and let us define for every $R>0$ the function $\phi_R(x):=R^d\phi(Rx)$. Then we let
\begin{equation}
    f^a_R(x):=(f1_{D_a})*\phi_R.
\end{equation}
It follows from standard properties of mollifiers that the following hold: firstly for every $a,R>0$ one has $f^a_R\in C^\infty(\mathbb{R}^d)$, secondly $f^a_R\to f1_{D_a}$, $\nabla f^a_R\to \nabla f1_{D_a}$ and $\Delta f^a_R\to \Delta f1_{D_a}$ pointwise on $D_a$. See for instance the relevant chapter in \cite{evans1998partial}. Furthermore by a standard argument involving the limit definition of a derivative, one can pass the derivative to either the mollifier or $f$, and therefore by the compact support of $\phi$ and the definition of $D_a$, it is straightforward to see for sufficiently large $R>0$ one has
\begin{equation}
    \lvert f^a_R\rvert,\lvert \nabla f^a_R\rvert, \lvert \Delta f^a_R\rvert \leq a,
\end{equation}
uniformly on $D_a$. As a result, we are able to apply the standard Itô's formula up to $t\wedge \tau\wedge \tau_a$ to obtain
\begin{align}
        f^a_R(Y_{t\wedge\tau\wedge \tau_a})&= f^a_R(y_0)+\int^{t\wedge\tau\wedge \tau_a}_0 (\langle \nabla f^a_R(Y_{s\wedge \tau_a}), g_s\rangle +\frac{1}{2}\Delta f^a_R(Y_{s\wedge \tau_a}))ds\nonumber \\
        &+\int^{t\wedge\tau\wedge \tau_a}_0 \langle \nabla f^a_R(Y_{s\wedge \tau_a}), dW_s\rangle.
    \end{align}
    Then by the boundedness of $f^a_R$ and its derivatives, using the dominated convergence theorem and the dominated convergence theorem for stochastic integrals, taking the $R\to\infty$ limit yields
    \begin{align}
        f(Y_{t\wedge\tau\wedge \tau_a})&=  f(y_0)+\int^{t\wedge\tau\wedge \tau_a}_0 (\langle \nabla  f(Y_{s\wedge \tau_a}), g_s\rangle +\frac{1}{2}\Delta  f(Y_{s\wedge \tau_a}))ds\nonumber \\
        &+\int^{t\wedge\tau\wedge \tau_a}_0 \langle \nabla  f(Y_{s\wedge \tau_a}), dW_s\rangle.
    \end{align}
We note that we have written $f$ not $f1_{D_a}$ in the above expression since $Y_{s\wedge \tau_a}\not\in D_a$ only occurs either at $s=0$, in which the above holds vacuously, or for $s\geq \tau_a$, which corresponds to a set of measure $0$ and therefore doesn't contribute to the integrals above. Now we must simply show that $\tau\wedge \tau_a\to\tau$ almost surely as $a\to \infty$. Suppose not. Then, since $\tau_a$ is almost surely increasing, there has to be a set of positive measure $A\subset \Omega$ such that for every $\omega\in A$ one has $\lim_{a\to\infty}\tau_a(\omega)=\tau_\infty(\tau)<\tau$. Since $Y_{t\wedge \tau}$ has almost surely continuous sample paths, for every $\omega \in A$ one furthermore has $Y_{\tau_{a_i}\wedge \tau}(\omega)\to y(\omega)\in D$. However, by the definition of $D_a$, for fixed $\omega\in A$ there must exist a sequence $y_i$ of points such that $\max\{\lvert f(y_i)\rvert,\lvert \nabla f(y_i)\rvert, \lvert \Delta f(y_i)\rvert\}=i$ and $\lvert y_i-Y_{\tau_i}(\omega)\rvert \leq i^{-1}$, so $y_i\to y$ also. Therefore by the continuity of $f\in C^2(D)$ on $D$, one cannot have $y\in D$. Therefore we have a contradiction, and the proof is completed.
\end{proof}
\section{Langlart Domination Inequality}
This result was first established in \cite{Lenglart1977}. 
\begin{prop}\label{prop: Langlart ineq}
    Let $Y_t, Z_t>0$ be $(\mathcal{F}_t)_{t\geq 0}$-measurable processes. Suppose for any bounded stopping time $\tau$ one has
    \begin{equation}
        E[Y_\tau \vert \mathcal{F}_0]\leq  E[Z_\tau \vert \mathcal{F}_0].
    \end{equation}
    Then for every $r\in (0,1)$ one has
    \begin{equation}
        E\sup_{t\geq 0}Y_t^r \leq \frac{r^{-r}}{1-r}E\sup_{t\geq 0}Z_t^r.
    \end{equation}
\end{prop}
\begin{proof}
    See Chapter IV, Proposition 4.7 in \cite{revuz2013continuous}.
\end{proof}

\section*{Acknowledgements}
The first author was supported by the ERC Synergy OCEAN grant (ERC-2022-SyG, 101071601). The second author is supported by the project CONVIVIALITY (ANR-23-CE40-0003) of the French National Research Agency.

This work was presented in a poster session at the Eurandom conference in Eindhoven, December 2025. We thank the attendants for their useful conversations. Thanks additionally to Christian Robert and Sotirios Sabanis.

\bibliographystyle{plain}

\bibliography{ref.bib}

\end{document}